\documentclass[12pt,centertags]{amsart}
\usepackage{amsmath,amstext,amsthm,a4,amssymb,amscd}
\usepackage[mathscr]{eucal}
\usepackage{mathrsfs}
\usepackage{epsf}
\numberwithin{equation}{section}

\def\mE{\mathcal{E}}
\def\mF{\mathcal{F}}

\def\mM{\mathcal{M}}
\def\mN{\mathcal{N}}

\newtheorem{thm}{Theorem}[section]
\newtheorem{lemma}[thm]{Lemma}
\newtheorem{prop}[thm]{Proposition}
\newtheorem{cor}[thm]{Corollary}
\theoremstyle{definition}
\newtheorem{rem}[thm]{Remark}
\theoremstyle{definition}
\newtheorem{ex}[thm]{Example}
\theoremstyle{definition}
\newtheorem{defn}[thm]{Definition}
\newcommand{\be}{\begin{eqnarray}}
\newcommand{\ee}{\end{eqnarray}}

\newcommand{\comment}[1]{}

\begin{document}

\title{Positive scalar curvature on foliations}
 
\author{Weiping Zhang}

\address{Chern Institute of Mathematics \& LPMC, Nankai
University, Tianjin 300071, P.R. China\\  
{\it and}\\
Center for Applied Mathematics,
Tianjin University,
Tianjin 300072,
P. R. China }
\email{weiping@nankai.edu.cn}

\begin{abstract}  We generalize classical theorems due to Lichnerowicz and Hitchin on the existence of Riemannian metrics of  positive scalar curvature on spin manifolds to the case of foliated spin manifolds.  As a consequence,  we show that  there is no foliation  of positive leafwise scalar curvature on any torus, which generalizes the famous theorem of Schoen-Yau and Gromov-Lawson on the non-existence of metrics of positive scalar curvature on torus  to the case of foliations. 
￼
Moreover, our method, which is partly inspired by the analytic localization techniques of Bismut-Lebeau,  
also applies to give a new proof of the celebrated Connes vanishing theorem without using noncommutative geometry. 
￼
\end{abstract}

\maketitle
\tableofcontents

\setcounter{section}{-1}
%%%%%%%%%%%%%%%%%%%%%%%%%%%%%%%%%%%%%%%%%%%
%%%%%%%%%%%%%%%%%%%%%%%%%%%%%%%%%%%%%%%%%%%

\section{Introduction} \label{s0}
%%%%%%%%%%%%%%%%%%%%%%%%%%%%%%%%%%%%%%%%%%%

It has been an important subject in differential  geometry  to study when a smooth manifold carries a Riemannian metric of positive scalar curvature (cf. \cite[Chap. IV]{LaMi89} and \cite{Gr96}).  In this paper, we   study related problems on foliations.

Let $F $ be an
integrable subbundle of the tangent vector bundle $TM$ of a smooth manifold $M$. For any
  Euclidean metric $g^F$ on $F$, let
$k^F\in C^\infty(M)$, which will be called the leafwise scalar curvature associated to $g^F$, be defined as follows: 
  for any $x\in M$,   the integrable subbundle
$F$ determines a leaf ${\mathcal F}_x$ passing through $x$ such
that $F|_{{\mathcal F}_x}=T{\mathcal F}_x$. Then, $g^F$ determines
a Riemannian metric on ${\mathcal F}_x$. Let $k^{{\mathcal F}_x}$
denote the scalar curvature of this Riemannian metric. We
define 
\begin{align}\label{001}
k^F(x)=k^{{\mathcal F}_x}(x).
\end{align}

For a closed spin manifold $M$, let
$\widehat{\mathcal A}(M)$ be the canonical $KO$-characteristic number of $M$ defined by that if $\dim M=8k+4i$
with $i=0$ or $1$, then $\widehat{\mathcal
A}(M)=\frac{3+(-1)^{i}}{4}\widehat{A}(M)$;\footnote{Cf. \cite[pp. 13]{Z01} for a
definition of the Hirzebruch $\widehat A$-genus $\widehat A(M)$.}  if $\dim M=8k+i$ with
$i=1$ or $2$, then $\widehat{\mathcal A}(M)\in {\bf Z}_2$ is the
Atiyah-Milnor-Singer $\alpha$ invariant;\footnote{Cf. \cite[\S
2.7]{LaMi89} for a definition.} while in other dimensions one
takes  $\widehat{\mathcal A}(M)=0$.

The main result of this paper can be stated as follows.

\begin{thm}\label{t0.2} Let $F $ be an  integrable subbundle  of the
tangent  bundle   of   a closed spin manifold $M$. If   $F$  
carries a metric of positive leafwise scalar curvature,  then
$\widehat{\mathcal A}(M)=0$.
\end{thm}

When   $F=TM$, one recovers the classical theorems due to  Lichnerowicz \cite{L63} (for the case of $\dim M=4k$)  and Hitchin \cite{Hi74} (for the cases of $\dim M=8k+1$ and $8k+2$).

\begin{ex}   
Take 
any $8k + 1$ dimensional closed spin manifold $M$ such that $\widehat {\mathcal A}(M)\neq 0$. By a result of Thurston \cite{T}, there always exists a codimension one foliation on $M$. However, by our result, there is no metric  of positive leafwise scalar curvature
on the associated integrable subbundle of $TM$.
\end{ex}

\begin{rem}\label{t002}
It  is a longstanding open question in foliation theory (cf. \cite[Remark
C14]{MS88}) that whether the existence of $g^F$ with $k^F>0$
implies the existence of $g^{TM}$ with $k^{TM}>0$.    This question admits an easy positive answer in the case where
$(M,F) $ carries a transverse  Riemannian structure (when such a transverse Riemannian structure exists,   $(M,F)$ is called a Riemannian foliation).  An approach  to this question for
codimension one foliations is outlined in the long paper of Gromov \cite[page 193]{Gr96}.
\end{rem}

Combining Theorem \ref{t0.2} with the well-known  results  of
Gromov-Lawson 
\cite{GL80a} 
and Stolz \cite{S92}, one gets the following consequence which provides a positive answer to the above question for simply connected manifolds of dimension greater than or equal to five.  

\begin{cor}\label{t0.6}  Let $F$ be an   integrable subbundle of the tangent bundle of a closed simply connected
manifold $M$ with $\dim M\geq 5$. If  $F$  carries a metric   of positive leafwise scalar curvature,  then $M$ admits  a Riemannian
metric of positive scalar curvature.
\end{cor}

For non-simply connected manifolds,  recall that a famous result due to Schoen-Yau \cite{SY79} and Gromov-Lawson \cite{GL80}    states that there is no metrics of positive scalar curvature on any torus. 
By combining Theorem \ref{t0.2} with the techniques
of Lusztig \cite{Lu71} and Gromov-Lawson \cite{GL80}, one obtains
the following   generalization to   the case
of foliations.

 \begin{cor}\label{t0.8}  There exists no foliation $(T^n,F)$ on any torus $T^n$ such
 that the integrable subbundle $F$ of $T(T^n)$ carries a metric of positive leafwise scalar curvature.
\end{cor}

If $F$ is further  assumed to be spin,   then Corollaries \ref{t0.6} and \ref{t0.8} can  also be deduced from the following celebrated vanishing theorem of Connes, which provides another kind of generalization 
of the Lichnerowicz theorem \cite{L63} to the case of foliations.

 \begin{thm}\label{t0.1} {\bf (Connes \cite[Theorem 0.2]{Co86})}   Let  $F$ be a spin integrable subbundle of  the tangent   bundle of a  compact
oriented manifold $M$. If   $F$ carries a metric of positive leafwise scalar curvature, 
  then  $\widehat{A}(M)=0$.
\end{thm}

Recall that the proof of Theorem
\ref{t0.1} outlined    in \cite{Co86}      makes use of  noncommutative geometry in an essential way. It is based on the
Connes-Skandalis longitudinal index theorem for foliations
\cite{CoS84} as well as the techniques of cyclic cohomology. Thus
it relies   on the spin structure on $F$, and we do not
see how to adapt it to prove Theorem \ref{t0.2}, where one assumes $TM$ being spin instead. 

On the other hand, while Theorem \ref{t0.2} is different from Connes' result and also covers the cases of $\dim M=8k+1$  and $8k+2$ where the Hirzebruch $\widehat A$-genus vanishes tautologically, a common difficulty  for both Theorems \ref{t0.2} and \ref{t0.1} is  that there might be no transverse Riemannian structure on the underlying foliated manifold. 

To overcome this difficulty, Connes \cite{Co86}  introduces  an important   geometric idea,  which reduces the original problem to  that on a
fibration\footnote{Which will be called a Connes fibration in what
follows.} over the foliation under consideration.  
The key advantage of this fibration is that the lifted (from the original) foliation is
almost isometric, i.e., very close to  Riemannian foliations. On the other hand, however,  this fibration is
noncompact. This makes the proof  of   Theorem 
\ref{t0.1}  in \cite{Co86}, which relies essentially on the noncommutative techniques, highly  nontrivial.

Our proof of Theorem \ref{t0.2} is     differential geometric and does not use  noncommutative geometry. It makes use of the sub-Dirac operators constructed in \cite[\S 2b)]{LZ01} on the Connes fibration, 
 as well as the adiabatic limit computations on foliations also considered  in \cite{LZ01}. 
The key point is that while Connes' noncommutative proof of Theorem \ref{t0.1} relies heavily on the analysis near the (fiberwise) infinity of the associated Connes fibration, our main concern is on a compact subset of the Connes fibration. To be more precise, inspired by \cite{BL91}, \cite{BZ93} and \cite{Co86}, we  introduce a specific   deformation of  the sub-Dirac operator on the Connes fibration    and show that the deformed operator is  ``invertible'' on  certain compact subsets of the Connes fibration  (cf. (\ref{2.151x})   in Section \ref{s2.3} for more details).

Moreover, by modifying the sub-Dirac operators mentioned above (see Section \ref{s1.4} for more details), 
our method   applies to give a purely  geometric proof of  Theorem \ref{t0.1}.  This new proof   provides a positive answer  to a longstanding   question
in index theory (cf. \cite[Page 5 of Lecture 9]{HR}).

We would  like to mention that the idea of constructing sub-Dirac
operators has also been used in \cite{LMZ01} to prove a
generalization of the Atiyah-Hirzebruch vanishing theorem for circle
actions \cite{AH} to the case of foliations.

This paper is organized as follows. In Section 1, we discuss the
case of almost isometric foliations and carry out the local
computations. We also introduce the sub-Dirac operator in this case
and prove Theorem \ref{t0.1} in the case where the underlying
foliation is compact. In Section 2, we work on   noncompact Connes
  fibrations and carry out the proofs of Theorems \ref{t0.2} and
\ref{t0.1}.  We also include some new results in the end of the paper.

%%%%%%%%%%%%%%%%%%%%%%%%%%%%%%%%%%%%%%%%%%%
%%%%%%%%%%%%%%%%%%%%%%%%%%%%%%%%%%%%%%%%%%%
\section{Adiabatic limit and almost isometric  foliations}\label{s1}
%%%%%%%%%%%%%%%%%%%%%%%%%%%%%%%%%%%%%%%%%%%

In this section, we discuss the geometry of almost isometric
foliations in the sense of Connes \cite{Co86}. We introduce for
this kind of foliations a rescaled  metric and show that 
the leafwise scalar curvature shows up
from the limit behavior of the rescaled  scalar curvature.
We also introduce in this setting the sub-Dirac operators inspired
by the original construction given in \cite{LZ01}.
Finally, by combining the above two procedures, we prove a vanishing
result when the almost isometric foliation under discussion is
compact.

This section is organized as follows. In Section \ref{s1.1}, we
recall the definition of the almost isometric foliation in the
sense of Connes. In Section \ref{s1.2} we introduce a rescaling of
the given metric on the almost isometric foliation  and study the
corresponding limit behavior of the scalar curvature. In Section
\ref{s1.3}, we study  Bott type connections on certain bundles
transverse  to the integrable subbundle. In Section \ref{s1.4}, we
construct the required  sub-Dirac operator  and compute the
corresponding Lichnerowicz type formula. In Section \ref{s1.5} we
prove a vanishing result when the almost isometric foliation is
compact and verifies the conditions in Theorem \ref{t0.1}.

%%%%%%%%%%%%%%%%%%%%%%%%%%%%%%%%%%%%%%%%%%%
\subsection{Almost isometric foliations
}\label{s1.1}

Let $(M,F)$ be a foliated manifold, where $F$ is an integrable
subbundle of the tangent vector bundle $TM$ of a smooth manifold $M$, i.e., for any smooth sections  $X,\ Y\in\Gamma(F)$, one has
\begin{align}\label{1.0}
[X,Y]\in \Gamma(F).
\end{align}

Take a splitting $TM=F\oplus TM/F$. 
Let $p^{TM/F}:TM=F\oplus TM/F\rightarrow TM/F$ be the canonical projection. Following \cite{Bo70}, we define the Bott connection to be any connection $\nabla^{TM/F}$ on $TM/F$ so that for any $X\in\Gamma(F)$ and $U\in\Gamma(TM/F)$, one has
\begin{align}\label{1.0a}
\nabla^{TM/F}_XU=p^{TM/F}[X,U]. 
\end{align} The key property of the Bott connection is that it is leafwise flat, that is, for any $X,\,Y\in \Gamma(F)$, one has (cf. \cite[Lemma 1.14]{Z01})
\begin{align}\label{1.0b}
\left(\nabla^{TM/F}\right)^2(X,Y)=0. 
\end{align}
However, it may happen that $\nabla^{TM/F}$ does not preserve any metric on $ TM/F$.

 Let $G$ be the holonomy groupoid of
$(M,F)$ (cf.
  \cite{W83}).

We make the assumption that
there is a proper subbundle $E$ of $TM/F$ and choose a splitting
\begin{align}\label{1.1} TM/F=E\oplus (TM/F)/E.
\end{align}
Let $q_1$, $q_2$ denote the ranks of $E$ and $(TM/F)/E$
respectively.

\begin{defn}\label{t1.1} {\bf (Connes \cite[Section 4]{Co86})} If there exists a metric $g^{TM/F}$ on
$TM/F$ with its restrictions to $E$ and $(TM/F)/E$ such that the
action of $G$ on $TM/F$ takes the form \begin{align}\label{1.2}
\left(
\begin{array}{cc}
O(q_1) & 0 \\
A & O(q_2)%
\end{array}%
\right),
\end{align}
where $O(q_1)$, $O(q_2)$ are orthogonal matrices of ranks $q_1$,
$q_2$ respectively, and $A$ is a $q_2\times q_1$ matrix, then we say
that $(M,F)$ carries an almost isometric structure.
\end{defn}

Clearly, the existence of the almost isometric structure does not
depend on the splitting (\ref{1.1}).
We assume from now on that $(M,F)$ carries an almost isometric structure as
above.

For simplicity, we denote $E$,
$(TM/F)/E$ by $F^\perp_1$, $F^\perp_2$
   respectively.

Let $g^F$ be a metric on $F$. Let
$g^{F^\perp_1}$, $g^{F^\perp_2}$ be the restrictions of
$g^{TM/F}$ to $F^\perp_1$, $F^\perp_2$.
Let $g^{TM}$ be a metric on $TM$ so that we have the orthogonal
splitting
\begin{align}\label{1.4}
       TM =  F\oplus F^\perp_1\oplus F^\perp_2,\ \ \ \ \ \ \
g^{TM} = g^{F}\oplus g^{F^\perp_1}\oplus
g^{F^\perp_2}.
\end{align}
Let $\nabla^{TM}$ be the Levi-Civita connection associated to
$g^{TM}$.

From the almost isometric condition (\ref{1.2}), one deduces that
for any $X\in \Gamma(F)$, $U_i,\, V_i\in\Gamma(F^\perp_i)$, $i=1,\
2$, the following identities, which may be thought of as infinitesimal versions of (\ref{1.2}),
 hold (cf. \cite[(A.5)]{LZ01}):
\begin{align}\label{1.5} \begin{split}
       \langle [X,U_i],V_i\rangle+\langle U_i,[X,V_i]\rangle =X\langle U_i,V_i\rangle,\\
\langle [X,U_2],U_1\rangle =0.\end{split}
\end{align}
Equivalently,
\begin{align}\label{1.6} \begin{split}
       \left\langle X,\nabla^{TM}_{U_i}V_i + \nabla^{TM}_{ V_i}U_i\right\rangle =0,\\
\left\langle \nabla^{TM}_XU_2 ,U_1\right\rangle+\left\langle
X,\nabla^{TM}_{U_2}U_1\right\rangle =0.\end{split}
\end{align}

In this paper, for simplicity, we also make the following assumption.
This assumption holds by the Connes   fibration to be dealt with in the next section.

\begin{defn}\label{t1.1c}  We call an almost isometric foliation as above verifies Condition ({C}) if  $F_2^\perp$ is also integrable. That is, for any $U_2,\, V_2\in\Gamma(F_2^\perp)$, one has
\begin{align}\label{1.6a}
\left[U_2,V_2\right]\in \Gamma\left(F_2^\perp\right).
\end{align}
\end{defn}

\subsection{Adiabatic limit and the scalar curvature
}\label{s1.2} In this subsection, we study the relationship between the leafwise scalar curvature and the scalar curvature on the total manifold of an almost isometric foliation. 
For convenience, we   recall the   formula for the
Levi-Civita connection (cf. \cite[(1.18)]{BeGeVe}) that for any $X,\ Y,\ Z\in\Gamma(TM)$,
\begin{multline}\label{1.7}
2\left\langle \nabla^{TM}_XY,Z\right\rangle =X\langle Y,Z\rangle
+Y\langle X,Z\rangle-Z\langle X,Y\rangle\\
+\langle [X,Y],Z\rangle -\langle [X,Z],Y\rangle-\langle
[Y,Z],X\rangle.
\end{multline}

 Recall that  by \cite[Proposition A.2]{LZ01}, if one rescales the metric $g^{F_1^\perp}$ to $\frac{1}{\varepsilon^2}g^{F_1^\perp}$ and takes $\varepsilon\rightarrow 0$, then  
the almost isometric foliation in the sense of Definition \ref{t1.1} becomes  an almost Riemannian foliation in
the sense of \cite[Definition 2.1]{LZ01}. In order to get information on the leafwise scalar curvature, one further rescales the metric $\frac{1}{\varepsilon^2}g^{F_1^\perp}\oplus g^{F_2^\perp}$ (standardly) to
$\frac{1}{\beta^2}(\frac{1}{\varepsilon^2}g^{F_1^\perp}\oplus g^{F_2^\perp})$  (compare with \cite[(1.4)]{LZ01} and \cite{LW09}), which is equivalent to rescaling $g^F$ to $\beta^2g^F$.  Putting these two rescaling procedures together, it is natural to introduce the following defomation of $g^{TM}$.

For any $\beta,\ \varepsilon>0$, let
$g_{\beta,\varepsilon}^{TM}$ be the rescaled  Riemannian metric on
$TM$ defined by
\begin{align}\label{1.8}
g^{TM}_{\beta,\varepsilon }= \beta^2 g^{F}\oplus \frac{1}{
\varepsilon^{2}} g^{F^\perp_1}\oplus    g^{F^\perp_2}.
\end{align}
We will always assume that $0<\beta,\ \varepsilon\leq 1$.

We will use the subscripts and/or superscripts ``$\beta$, $\varepsilon$'' to
decorate the geometric data associated to $g^{TM}_{\beta,\varepsilon }$.
For example, $\nabla^{TM,\beta,\varepsilon}$ will denote the Levi-Civita
connection associated to $g^{TM}_{\beta,\varepsilon} $. When the
corresponding notation does not involve ``$\beta,\ \varepsilon$'', we will
mean that it corresponds to the case of $\beta=\varepsilon=1$.

Let $p$, $p_1^\perp$, $p_2^\perp$ be the orthogonal projections from
$TM$ to $F$, $F_1^\perp$, $F^\perp_2$ with respect to the orthogonal
splitting (\ref{1.4}).
Let $\nabla^{F,\beta,\varepsilon}$, $\nabla^{F_1^\perp,\beta,\varepsilon}$,
$\nabla^{F_2^\perp,\beta,\varepsilon}$ be the Euclidean connections on
$F$, $F_1^\perp$, $F^\perp_2$ defined by
\begin{align}\label{1.9}
\nabla^{F,\beta,\varepsilon}=p \nabla^{TM,\beta,\varepsilon}p,\ \
\nabla^{F_1^\perp,\beta,\varepsilon}=p_1^\perp
\nabla^{TM,\beta,\varepsilon}p_1^\perp,\ \
\nabla^{F_2^\perp,\beta,\varepsilon}=p_2^\perp\nabla^{TM,\beta,\varepsilon}p_2^\perp.
\end{align}
In particular, one has
\begin{align}\label{1.9a}
\nabla^{F }=p \nabla^{TM }p,\ \
\nabla^{F_1^\perp }=p_1^\perp
\nabla^{TM }p_1^\perp,\ \
\nabla^{F_2^\perp }=p_2^\perp\nabla^{TM }p_2^\perp.
\end{align}

By (\ref{1.7})-(\ref{1.9a}) and the integrability of $F$, the
following identities hold for $X\in\Gamma(F)$:
\begin{align}\label{1.10}
\nabla^{F,\beta,\varepsilon}=\nabla^F ,\ \ p
\nabla^{TM,\beta,\varepsilon}_Xp_i^\perp=p \nabla^{TM }_Xp_i^\perp,\ \
i=1,\ 2,
\end{align}
$$p_1^\perp \nabla^{TM,\beta,\varepsilon}_X p=\beta^{2 }\varepsilon^{2} p_1^\perp \nabla^{TM }_X p,\ \ \
 p_2^\perp \nabla^{TM,\beta,\varepsilon}_X p=\beta^2   p_2^\perp \nabla^{TM }_X p. $$

 From (\ref{1.5})-(\ref{1.8}), we deduce that for
 $X\in\Gamma(F)$, $U_i,\, V_i\in\Gamma(F^\perp_i)$, $i=1,\, 2$,
\begin{align}\label{1.11}
 \left\langle \nabla^{TM,\beta,\varepsilon}_{U_1}V_1,X\right\rangle = \left\langle \nabla^{TM
 }_{U_1}V_1,X\right\rangle= \frac{1}{2 }  \left\langle \left[{U_1},V_1\right],X\right\rangle,
\end{align}
while
\begin{align}\label{1.11a}
 \left\langle \nabla^{TM,\beta,\varepsilon}_{U_2}V_2,X\right\rangle = \left\langle \nabla^{TM
 }_{U_2}V_2,X\right\rangle= \frac{1}{2 }  \left\langle \left[{U_2},V_2\right],X\right\rangle=0.
\end{align}
Equivalently, for any $U_i\in \Gamma(F^\perp_i)$, $i=1,\, 2$,
\begin{align}\label{1.12}
 p_1^\perp\nabla^{TM,\beta,\varepsilon}_{U_1}p= {\beta^{2 }\varepsilon^{2}}  p_1^\perp\nabla^{TM
 }_{U_1}p,\ \ \ p_2^\perp\nabla^{TM,\beta,\varepsilon}_{U_2}p=0.
\end{align}
Similarly, one verifies that
\begin{align}\label{1.13}
  \left\langle \nabla^{TM,\beta,\varepsilon}_{U_1}X,U_2\right\rangle
  =\frac{1}{2} \left\langle[U_1,X],U_2\right\rangle
  -\frac{\beta^2 }{2}\left\langle[U_1,U_2],X\right\rangle,
\end{align}
$$\left\langle \nabla^{TM,\beta,\varepsilon}_{U_2}X,U_1\right\rangle
  =\frac{\varepsilon^{2 }}{2 } \left\langle[U_1,X],U_2\right\rangle
  +\frac{\beta^{2 }\varepsilon^{2}}{2}\left\langle[U_1,U_2],X\right\rangle .$$

 For   convenience  of the  later computations, we collect the
 asymptotic behavior of various
covariant derivatives in the following lemma. These formulas can
be derived by applying (\ref{1.5})-(\ref{1.13}). The inner products
appear in the   lemma correspond to $\beta=\varepsilon=1$.

\begin{lemma}\label{tf} The following formulas hold for $X,\,
Y,\,Z\in \Gamma(F)$, $U_i,\,V_i,\, W_i\in \Gamma(F_i^\perp)$ with
$i=1,\, 2$, when $\beta>0$, $\varepsilon>0$ are small,
\begin{align}\label{f1}
 \left\langle\nabla^{TM,\beta,\varepsilon}_XY,Z\right\rangle  =O(1)
 , \ \
 \left\langle\nabla^{TM,\beta,\varepsilon}_XY,U_1\right\rangle=O\left(\beta^{2 }\varepsilon^{2}\right),\ \
\left\langle\nabla^{TM,\beta,\varepsilon}_XY,U_2\right\rangle=O\left(\beta^2\right),
\end{align}
\begin{align}\label{f2}
 \left\langle\nabla^{TM,\beta,\varepsilon}_XU_1,Y\right\rangle  =O\left(1\right)
 , \ \
 \left\langle\nabla^{TM,\beta,\varepsilon}_XU_1,V_1\right\rangle =O\left(1 \right),\ \
\left\langle\nabla^{TM,\beta,\varepsilon}_XU_1,U_2\right\rangle
=O\left(1\right),
\end{align}
\begin{align}\label{f3}
 \left\langle\nabla^{TM,\beta,\varepsilon}_XU_2,Y\right\rangle  =O\left(1\right)
 , \ \
 \left\langle\nabla^{TM,\beta,\varepsilon}_XU_2,U_1\right\rangle =O\left(  {\varepsilon^{2} } \right),\ \
\left\langle\nabla^{TM,\beta,\varepsilon}_XU_2,V_2\right\rangle =O\left(
{1} \right),
\end{align}
\begin{align}\label{f4}
 \left\langle\nabla^{TM,\beta,\varepsilon}_{U_1}X,Y\right\rangle  =O\left(1\right)
 , \ \
 \left\langle\nabla^{TM,\beta,\varepsilon}_{U_1}X,V_1\right\rangle =O\left( {\beta^{ 2}\varepsilon^{2}} \right),\ \
\left\langle\nabla^{TM,\beta,\varepsilon}_{U_1}X,U_2\right\rangle =O\left(
{1} \right),
\end{align}
\begin{align}\label{f5}
 \left\langle\nabla^{TM,\beta,\varepsilon}_{U_1}V_1,X\right\rangle  =O\left( {1} \right)
 , \ \
 \left\langle\nabla^{TM,\beta,\varepsilon}_{U_1}V_1,W_1\right\rangle =O\left( {1} \right),\ \
\left\langle\nabla^{TM,\beta,\varepsilon}_{U_1}V_1,U_2\right\rangle
=O\left(\frac{1}{\varepsilon^{2}}\right),
\end{align}
\begin{align}\label{f6}
 \left\langle\nabla^{TM,\beta,\varepsilon}_{U_1}U_2,X\right\rangle  =O\left(\frac{1}{\beta^2}\right)
 , \ \
 \left\langle\nabla^{TM,\beta,\varepsilon}_{U_1}U_2,V_1\right\rangle =O\left( {1}  \right),\ \
\left\langle\nabla^{TM,\beta,\varepsilon}_{U_1}U_2,V_2\right\rangle
=O\left(1\right),
\end{align}
\begin{align}\label{f7}
 \left\langle\nabla^{TM,\beta,\varepsilon}_{U_2}X,Y\right\rangle =O\left( {1} \right)
 , \ \
 \left\langle\nabla^{TM,\beta,\varepsilon}_{U_2}X,U_1\right\rangle =O\left( { {\varepsilon^{2}} }
\right),\ \
\left\langle\nabla^{TM,\beta,\varepsilon}_{U_2}X,V_2\right\rangle
=0,
\end{align}
\begin{align}\label{f8}
 \left\langle\nabla^{TM,\beta,\varepsilon}_{U_2}U_1,X\right\rangle  =O\left( \frac{1}{\beta^2} \right)
 , \ \
 \left\langle\nabla^{TM,\beta,\varepsilon}_{U_2}U_1,V_1\right\rangle =O\left( {1}  \right),\ \
\left\langle\nabla^{TM,\beta,\varepsilon}_{U_2}U_1,V_2\right\rangle
=O\left(1\right),
\end{align}
\begin{align}\label{f9}
 \left\langle\nabla^{TM,\beta,\varepsilon}_{U_2}V_2,X\right\rangle  =0
 , \ \
 \left\langle\nabla^{TM,\beta,\varepsilon}_{U_2}V_2,U_1\right\rangle =O\left( {\varepsilon^2}  \right),\ \
\left\langle\nabla^{TM,\beta,\varepsilon}_{U_2}V_2,W_2\right\rangle
=O\left( {1} \right).
\end{align}
\end{lemma}

\begin{proof}
Formulas in (\ref{f1}) follows from (\ref{1.10}). 

The first formula in (\ref{f2}) follows from (\ref{1.8}) and the second formula in (\ref{f1}). The second one is trivial and the third one follows from (\ref{1.13}). 

The first formula in (\ref{f3}) follows from (\ref{1.8}) and the third formula in (\ref{f1}).  The second one follows from the second formulas in (\ref{1.5}) and (\ref{1.13}). The third one is trivial. 

The first formula in (\ref{f4}) follows from (\ref{1.0}), (\ref{1.7}) and (\ref{1.8}). The second one follows from (\ref{1.12}) and the third one follows from the first formula in (\ref{1.13}). 

The first formula in (\ref{f5}) follows from (\ref{1.8}) and the second formula in (\ref{f4}). The second formula is trivial. For the third formula, the $\frac{1}{\varepsilon^2}$ factor comes from the terms involving $\langle [U_1,U_2],V_1\rangle$, $\langle [V_1,U_2],U_1\rangle$ and $U_2\langle U_1,V_1\rangle$.

The first formula in (\ref{f6}) follows from the first formula in (\ref{1.13}). The second one is trivial, and the third one follows from (\ref{1.6a}). 

The first formula in (\ref{f7}) follows from the first formula in (\ref{1.10}). The second one follows from the second formula in (\ref{1.13}), and third one follows from (\ref{1.11a}). 

The first formula in (\ref{f8})  follows from (\ref{1.8}) and the second formula in (\ref{f7}). 
The second one is trivial, and the third one follows from (\ref{1.6a}).

The first formula in (\ref{f9}) follows from the third formula in (\ref{f7}). The second one follows from the third formula in (\ref{f8}), and the third one is trivial.

The proof of Lemma \ref{tf} is completed. 
\end{proof}

In what follows, when we compute the asymptotics of various covariant derivatives,  
we will simply use the above asymptotic formulas freely without further notice.

  Let $R^{TM,\beta,\varepsilon}=(\nabla^{TM,\beta,\varepsilon})^2$ be the curvature of
  $\nabla^{TM,\beta,\varepsilon}$. Then for any $X,\ Y\in\Gamma(TM)$, one has
  the following standard formula,
\begin{align}\label{1.14}
   R^{TM,\beta,\varepsilon}(X,Y)=\nabla^{TM,\beta,\varepsilon}_X\nabla^{TM,\beta,\varepsilon}_Y-\nabla^{TM,\beta,\varepsilon}_Y\nabla^{TM,\beta,\varepsilon}_X-\nabla^{TM,\beta,\varepsilon}_{[X,Y]}.
\end{align}
Let $R^F=(\nabla^F)^2$ be the curvature of $\nabla^F$.
Let $k^{TM,\beta,\varepsilon}$, $k^F$ denote the scalar curvatures
of $g^{TM,\beta,\varepsilon}$, $g^F$ respectively. Recall that
$k^F$ is defined in (\ref{001}). The following formula for $k^F$
is obvious,
\begin{align}\label{1.23aa}
 k^F=- \sum_{i,\, j=1}^{{\rm rk}(F)}\left\langle R^F\left(f_i,f_j\right)f_i,f_j\right\rangle,
\end{align}
where $f_i$, $i=1,\,\cdots,\, {\rm rk}(F)$, is an orthonormal basis of $(F,g^F)$.  Clearly, when $F=TM$, it reduces to the usual definition of the
scalar curvature $k^{TM}$ of $g^{TM}$.

\begin{prop}\label{t12.1}  If Condition (C) holds, then   when $\beta>0$, $\varepsilon>0$ are small, the following formula holds uniformly on any compact subset of $M$,
\begin{align}\label{1.24}
k^{TM,\beta,\varepsilon}=\frac{k^F }{\beta^2}+O\left(1+\frac{\varepsilon^2}{\beta^2}\right).
\end{align}
\end{prop}
\begin{proof}
 By (\ref{1.0}),  (\ref{1.10}),
(\ref{1.14}) and Lemma \ref{tf}, one
deduces that when $\beta>0$, $\varepsilon>0$ are very small, for any $X,\
Y\in\Gamma(F)$, one has
\begin{multline}\label{1.15}
\left\langle
R^{TM,\beta,\varepsilon}(X,Y)X,Y\right\rangle=\left\langle\nabla^{TM,\beta,\varepsilon}_X
\left(p+p_1^\perp+p_2^\perp\right)\nabla^{TM,\beta,\varepsilon}_YX,Y\right\rangle
\\
-\left\langle\nabla^{TM,\beta,\varepsilon}_Y\left(p+p_1^\perp+p_2^\perp\right)
\nabla^{TM,\beta,\varepsilon}_XX,Y\right\rangle-\left\langle
\nabla^{TM,\beta,\varepsilon}_{[X,Y]}X,Y\right\rangle\\
=\left\langle R^F(X,Y)X,Y\right\rangle
-\beta^{2 }\varepsilon^{2}\left\langle p_1^\perp
\nabla^{TM}_YX,\nabla^{TM}_XY\right\rangle
-\beta^2\left\langle p_2^\perp
\nabla^{TM}_YX,\nabla^{TM}_XY\right\rangle\\
+\beta^{2 }\varepsilon^{2}\left\langle p_1^\perp
\nabla^{TM}_XX,\nabla^{TM}_YY\right\rangle+\beta^2\left\langle
p_2^\perp \nabla^{TM}_XX,\nabla^{TM}_YY\right\rangle\\
=\left\langle R^F(X,Y)X,Y\right\rangle
+O\left( \beta^2\right).
\end{multline}

For $X\in\Gamma(F),\ U\in\Gamma(F_1^\perp)$, by
(\ref{1.5})-(\ref{1.14}), one finds that when $\beta , \ \varepsilon>0$ are
  small,
\begin{multline}\label{1.16}
\left\langle
R^{TM,\beta,\varepsilon}(X,U)X,U\right\rangle=\left\langle\nabla^{TM,\beta,\varepsilon}_X
\left(p+p_1^\perp+p_2^\perp\right)\nabla^{TM,\beta,\varepsilon}_UX,U\right\rangle
\\
-\left\langle\nabla^{TM,\beta,\varepsilon}_U
\left(p+p_1^\perp+p_2^\perp\right)\nabla^{TM,\beta,\varepsilon}_XX,U\right\rangle-\left\langle
\nabla^{TM,\beta,\varepsilon}_{\left(p+p_1^\perp+p_2^\perp\right)[X,U]}X,U\right\rangle\end{multline}
$$
=\beta^{2 }\varepsilon^{2}\left\langle\nabla^{TM}_Xp\nabla^{TM}_UX,U\right\rangle+\beta^{2 }\varepsilon^{2}
\left\langle\nabla^{TM,\beta,\varepsilon}_Xp_1^\perp
\nabla^{TM}_UX,U\right\rangle - {\varepsilon^{2}} \left\langle
p_2^\perp\nabla^{TM,\beta,\varepsilon}_UX,
\nabla^{TM,\beta,\varepsilon}_XU\right\rangle
$$
$$
-\beta^{2 }\varepsilon^{2}\left\langle\nabla^{TM}_Up\nabla^{TM}_XX,U\right\rangle-
\beta^{2 }\varepsilon^{2}
\left\langle\nabla^{TM,\beta,\varepsilon}_Up_1^\perp\nabla^{TM}_XX,U\right\rangle
+ {\varepsilon^{2}} \left\langle
p_2^\perp\nabla^{TM,\beta,\varepsilon}_XX,\nabla^{TM,\beta,\varepsilon}_UU\right\rangle
$$
$$
-\beta^{2 }\varepsilon^{2}\left\langle\nabla^{TM}_{\left(p+p_1^\perp\right)[X,U]}X,U\right\rangle - \left\langle\nabla^{TM,\beta,\varepsilon}_{ p_2^\perp [X,U]}X,U\right\rangle
= O\left(\beta^2+  \varepsilon^{2} \right).$$

Similarly,  for $X\in\Gamma(F)$,
$U\in\Gamma(F_2^\perp)$, one has that when $\beta>0$, $\varepsilon>0$ are small,
\begin{multline}\label{1.18}
\left\langle
R^{TM,\beta,\varepsilon}(X,U)X,U\right\rangle=\left\langle\nabla^{TM,\beta,\varepsilon}_X
\left(p+p_1^\perp+p_2^\perp\right)\nabla^{TM,\beta,\varepsilon}_UX,U\right\rangle
\\
-\left\langle\nabla^{TM,\beta,\varepsilon}_U
\left(p+p_1^\perp+p_2^\perp\right)\nabla^{TM,\beta,\varepsilon}_XX,U\right\rangle-\left\langle
\nabla^{TM,\beta,\varepsilon}_{\left(p+p_1^\perp+p_2^\perp\right)[X,U]}X,U\right\rangle\\
=\beta^2 \left\langle\nabla^{TM}_Xp\nabla^{TM}_UX,U\right\rangle-\frac{1}{\varepsilon^{2}}\left\langle
p_1^\perp \nabla^{TM,\beta,\varepsilon}_UX,
\nabla^{TM,\beta,\varepsilon}_XU\right\rangle+\beta^2\left\langle
\nabla^{TM,\beta,\varepsilon}_X
p_2^\perp\nabla^{TM}_UX,U\right\rangle
\end{multline}
$$
-\beta^2\left\langle\nabla^{TM}_Up\nabla^{TM}_XX,U\right\rangle- {\beta^{2 }}{\varepsilon^2}
\left\langle\nabla^{TM,\beta,\varepsilon}_Up_1^\perp\nabla^{TM}_XX,U\right\rangle
-\beta^2
\left\langle\nabla^{TM,\beta,\varepsilon}_Up_2^\perp\nabla^{TM}_XX,U\right\rangle$$
$$-\beta^2\left\langle\nabla^{TM}_{p[X,U]}X,U\right\rangle
-\beta^2\left\langle\nabla^{TM
}_{p_2^\perp[X,U]}X,U\right\rangle
=
O\left(  {\beta^2+\varepsilon^2} \right).$$

 For $U,\ V\in\Gamma(F_1^\perp)$, one verifies   that
\begin{multline}\label{1.19}
\left\langle
R^{TM,\beta,\varepsilon}(U,V)U,V\right\rangle=\left\langle\nabla^{TM,\beta,\varepsilon}_U
\left(p+p_1^\perp+p_2^\perp\right)\nabla^{TM,\beta,\varepsilon}_VU,V\right\rangle
\\
-\left\langle\nabla^{TM,\beta,\varepsilon}_V\left(p+p_1^\perp+p_2^\perp\right)
\nabla^{TM,\beta,\varepsilon}_UU,V\right\rangle-\left\langle
\nabla^{TM,\beta,\varepsilon}_{\left(p+p_1^\perp+p_2^\perp\right)[U,V]}U,V\right\rangle
\end{multline}
$$
=\beta^2\varepsilon^{2}\left\langle\nabla^{TM
}_Up\nabla^{TM,\beta,\varepsilon}_VU,V\right\rangle+\left\langle\nabla^{TM
}_Up_1^\perp\nabla^{TM
}_VU,V\right\rangle-\varepsilon^{2}\left\langle
p_2^\perp\nabla^{TM,\beta,\varepsilon}_VU,\nabla^{TM,\beta,\varepsilon}_UV\right\rangle$$
$$
-\beta^2\varepsilon^{2}\left\langle\nabla^{TM
}_Vp\nabla^{TM,\beta,\varepsilon}_UU,V\right\rangle
-\left\langle\nabla^{TM }_Vp_1^\perp\nabla^{TM }_UU,V\right\rangle
+\varepsilon^{2}\left\langle
p_2^\perp\nabla^{TM,\beta,\varepsilon}_UU,\nabla^{TM,\beta,\varepsilon}_VV\right\rangle$$
$$
-\left\langle\nabla^{TM,\beta,\varepsilon}_{p[U,V]}U,V\right\rangle-\left\langle\nabla^{TM
}_{p_1^\perp[U,V]}U,V\right\rangle-
\left\langle\nabla^{TM,\beta,\varepsilon}_{p_2^\perp[U,V]}U,V\right\rangle$$
$$
=-\varepsilon^{2}\left\langle
p_2^\perp\nabla^{TM,\beta,\varepsilon}_VU,\nabla^{TM,\beta,\varepsilon}_UV\right\rangle
+\varepsilon^{2}\left\langle
p_2^\perp\nabla^{TM,\beta,\varepsilon}_UU,\nabla^{TM,\beta,\varepsilon}_VV\right\rangle
 +O\left(1\right) =O\left(\frac{1}{\varepsilon^{2}}\right),$$
from which one gets that when $\beta>0$, $\varepsilon>0$ are
small,
\begin{align}\label{1.20}
\varepsilon^{2}\left\langle
R^{TM,\beta,\varepsilon}(U,V)U,V\right\rangle=O\left(1\right).
\end{align}

For $U,\ V\in \Gamma(F_2^\perp)$, one verifies directly that
\begin{multline}\label{1.21}
\left\langle
R^{TM,\beta,\varepsilon}(U,V)U,V\right\rangle=\left\langle\nabla^{TM,\beta,\varepsilon}_U\left(p+p_1^\perp+p_2^\perp\right)\nabla^{TM,\beta,\varepsilon}_VU,V\right\rangle
\\
-\left\langle\nabla^{TM,\beta,\varepsilon}_V\left(p+p_1^\perp+p_2^\perp\right)
\nabla^{TM,\beta,\varepsilon}_UU,V\right\rangle-\left\langle
\nabla^{TM,\beta,\varepsilon}_{ [U,V]}U,V\right\rangle\\
=\beta^2\left\langle\nabla^{TM
}_Up\nabla^{TM,\beta, \varepsilon}_VU,V\right\rangle-\frac{1}{\varepsilon^{2}}\left\langle
p_1^\perp\nabla^{TM,\beta,\varepsilon }_VU,\nabla^{TM,\beta,\varepsilon
}_UV\right\rangle+ \left\langle
\nabla^{TM }_Up_2^\perp\nabla^{TM }_VU,V\right\rangle\\
-\beta^2\left\langle\nabla^{TM
}_Vp\nabla^{TM,\beta,\varepsilon}_UU,V\right\rangle +
\frac{1}{\varepsilon^{2}}\left\langle
p_1^\perp\nabla^{TM,\beta,\varepsilon }_UU,\nabla^{TM,\beta,\varepsilon
}_VV\right\rangle -\left\langle
\nabla^{TM }_Vp_2^\perp\nabla^{TM }_UU,V\right\rangle\\
 -\left\langle
\nabla^{TM }_{ [U,V]}U,V\right\rangle=O(1).
\end{multline}

For $U\in \Gamma(F_1^\perp),\ V\in \Gamma(F_2^\perp)$, one verifies
directly that,
\begin{multline}\label{1.23}
\left\langle
R^{TM,\beta,\varepsilon}(U,V)U,V\right\rangle=\left\langle\nabla^{TM,\beta,\varepsilon}_U
\left(p+p_1^\perp+p_2^\perp\right)\nabla^{TM,\beta,\varepsilon}_VU,V\right\rangle
\\
-\left\langle\nabla^{TM,\beta,\varepsilon}_V\left(p+p_1^\perp+p_2^\perp\right)
\nabla^{TM,\beta,\varepsilon}_UU,V\right\rangle-\left\langle
\nabla^{TM,\beta,\varepsilon}_{ [U,V]}U,V\right\rangle\\
= -\beta^2\left\langle
p\nabla^{TM,\beta,\varepsilon}_VU,\nabla^{TM,\beta,\varepsilon}_UV\right\rangle-\frac{1}{\varepsilon^{2}}\left\langle
p_1^\perp\nabla^{TM,\beta,\varepsilon}_VU,\nabla^{TM,\beta,\varepsilon}_UV\right\rangle
+\left\langle\nabla^{TM,\beta,\varepsilon}_U p_2^\perp
\nabla^{TM,\beta,\varepsilon}_VU,V\right\rangle\\
+\beta^2\left\langle
p\nabla^{TM,\beta,\varepsilon}_UU,\nabla^{TM,\beta,\varepsilon}_VV\right\rangle+\frac{1}{\varepsilon^{2}}\left\langle
p_1^\perp\nabla^{TM,\beta,\varepsilon}_UU,\nabla^{TM,\beta,\varepsilon}_VV\right\rangle
-\left\langle\nabla^{TM }_V p_2^\perp
\nabla^{TM,\beta,\varepsilon}_UU,V\right\rangle\\
+\frac{1}{\varepsilon^{2}}\left\langle
U,\nabla^{TM,\beta,\varepsilon}_{[U,V]}V\right\rangle =O\left(\frac{1}{\varepsilon^2}+\frac{1}{\beta^2}\right) ,
\end{multline}
 from which  one gets that when $\beta>0$, $\varepsilon>0$ are small,
\begin{align}\label{1.23a}
\varepsilon^{2}\left\langle
R^{TM,\beta,\varepsilon}(U,V)U,V\right\rangle= \left\langle
R^{TM,\beta,\varepsilon}(V,U)V,U\right\rangle= O\left(1+\frac{\varepsilon^2}{\beta^2}\right).
\end{align}

From  (\ref{1.23aa}), (\ref{1.15})-(\ref{1.18}), (\ref{1.20}),  (\ref{1.21}) and
(\ref{1.23a}), one gets  (\ref{1.24}).  \end{proof}

\subsection{Bott connections on $F_1^\perp$ and $F_2^\perp$
}\label{s1.3}

From (\ref{1.5}) and  (\ref{1.6a})-(\ref{1.9}), one verifies directly
that for $X\in\Gamma(F)$, $U_i,\, V_i\in\Gamma(F_i^\perp)$, $i=1,\,
2$, one has
\begin{align}\label{1.25} \left\langle\nabla^{F_1^\perp,\beta,
\varepsilon}_XU_1,V_1
\right\rangle=\left\langle\left[X,U_1\right],V_1\right\rangle-
\frac{\beta^2\varepsilon^{2}}{2}\left\langle\left[U_1,V_1\right],X\right\rangle
,
\end{align}
$$\left\langle\nabla^{F_2^\perp,\beta,\varepsilon}_XU_2,V_2
\right\rangle=\left\langle\left[X,U_2\right],V_2\right\rangle.
$$

By (\ref{1.25}), one has that for $X\in\Gamma(F)$, $U_i
\in\Gamma(F_i^\perp)$, $i=1,\, 2$,
\begin{align}\label{1.26}\lim_{\varepsilon\rightarrow 0^+} \nabla^{F_i^\perp,\beta,
\varepsilon}_XU_i
 =
 \widetilde{\nabla}^{F_i^\perp }_XU_i:=p_i^\perp \left[X,U_i\right]
.
\end{align}

Let $\widetilde{\nabla}^{F_i^\perp }$ be the connection on
$F_i^\perp$ defined  by the second equality in (\ref{1.26}) and by
$\widetilde{\nabla}^{F_i^\perp }_UU_i= {\nabla}^{F_i^\perp }_UU_i$
for $U\in  \Gamma(F_1^\perp\oplus F_2^\perp)$. In
view of (\ref{1.0a}) and  (\ref{1.26}), we   call
$\widetilde{\nabla}^{F_i^\perp }$ a Bott connection on $F_i^\perp$
for   $i=1$ or $ 2$. Let $\widetilde{R}^{F_i^\perp }$ denote the
curvature of $\widetilde{\nabla}^{F_i^\perp }$ for $i=1,\, 2$.

The following result holds without Condition (C). 

\begin{lemma}\label{t1.2} For $X,\, Y\in\Gamma(F)$ and $i=1,\, 2$, the following
identity holds,
\begin{align}\label{1.27}
 \widetilde{R}^{F_i^\perp }(X,Y)=0
.
\end{align}
\end{lemma}
\begin{proof}We proceed as in \cite[Proof of Lemma 1.14]{Z01}. By
(\ref{1.26}) and the standard formula for the curvature (cf.
\cite[(1.3)]{Z01}), for any
$U\in\Gamma(F_i^\perp)$, $i=1,\, 2$,  one has,
\begin{multline}\label{1.28}
\widetilde{R}^{F_i^\perp }(X,Y)U=\widetilde{\nabla}^{F_i^\perp
}_X\widetilde{\nabla}^{F_i^\perp }_YU-\widetilde{\nabla}^{F_i^\perp
}_Y\widetilde{\nabla}^{F_i^\perp }_XU-\widetilde{\nabla}^{F_i^\perp
}_{[X,Y]}U\\
=p_i^\perp\big([X,[Y,U]]+[Y,[U,X]]+[U,[X,Y]]\big) -p_i^\perp
\left[X,\left({\rm Id}-p_i^\perp\right)[Y,U]\right]\\ -p_i^\perp
\left[Y,\left({\rm Id}-p_i^\perp\right)[U,X]\right]\\
=-p_i^\perp
\left[X,\left(p_1^\perp+p_2^\perp-p_i^\perp\right)[Y,U]\right]-p_i^\perp
\left[Y,\left(p_1^\perp+p_2^\perp-p_i^\perp\right)[U,X]\right],
\end{multline}
where the last equality follows from the Jacobi identity and the
integrability of $F$.

Now if $i=1$, then by (\ref{1.5}), one has $U\in\Gamma(F_1^\perp)$ and
\begin{align}\label{1.29}
p_1^\perp \left[X,  p_2^\perp  [Y,U]\right]=p_1^\perp \left[Y,
p_2^\perp  [U,X]\right]=0.
\end{align}
While if $i=2$, still by (\ref{1.5}), one has $U\in\Gamma(F_2^\perp)$ and
\begin{align}\label{1.30}
p_1^\perp[Y,U]=p_1^\perp[U,X]=0
 .
\end{align}

From (\ref{1.28})-(\ref{1.30}), one gets (\ref{1.27}). The proof of
Lemma  \ref{t1.2} is completed.
\end{proof}

\begin{rem}\label{t1.3} For $i=1,\, 2$,  let $ {R}^{F_i^\perp,\beta,\varepsilon }$ denote the
curvature of $ {\nabla}^{F_i^\perp,\beta,\varepsilon }$. From (\ref{1.25})-(\ref{1.27}), one finds that for any $X,\, Y\in\Gamma(F)$, when
$\beta>0$, $\varepsilon>0$ are small, the following identity holds:
\begin{align}\label{1.31}
 {R}^{F_i^\perp,\beta,\varepsilon }(X,Y)=O\left(\beta^2\varepsilon^2\right)
 .
\end{align}

On the other hand, for $i=1,\, 2$,  by using (\ref{1.5}), (\ref{1.6a}),
(\ref{1.7}), (\ref{1.9}) and (\ref{1.14}), one verifies directly
that when $\beta>0$, $\varepsilon>0$ are small, the following identity holds,
\begin{align}\label{1.101}
{R}^{F_i^\perp,\beta,\varepsilon } =O\left(1\right)
 .
\end{align}\end{rem}

%%%%%%%%%%%%%%%%%%%%%%%
\subsection{Sub-Dirac operators associated to spin integrable  subbundles }\label{s1.4} We assume  for simplicity 
 that $TM$, $ F$, $ F_i^\perp$, $i=1,\, 2$, are all oriented and
of even rank, with the orientation of $TM$ being  compatible with
the orientations on $F$, $F_1^\perp$ and $F_2^\perp$ through
(\ref{1.4}).
We further assume that $F$ is spin and carries a fixed spin
structure.

Let $S(F)=S_+(F)\oplus S_-(F)$ be the Hermitian bundle of spinors
associated to $(F,g^F)$. For any $X\in\Gamma(F)$, the Clifford
action $c(X)$ exchanges $S_\pm(F)$.

Let $i=1$ or $ 2$.
Let $\Lambda^*(F_i^{\perp })$ denote  the exterior algebra bundle
of $F_i^{\perp, *}$. Then $\Lambda^*(F_i^{\perp })$ carries a
canonically induced metric $g^{\Lambda^*(F_i^{\perp })}$ from
$g^{F_i^\perp}$. For any $U\in F_i^\perp$, let $U^*\in
F^{\perp,*}_i$ correspond  to $U$ via $g^{F_i^\perp}$.
For any $U\in\Gamma(F_i^\perp)$, set
\begin{align}\label{1.32}
c(U)=U^*\wedge-i_U,\ \ \ \widehat{c}(U)=U^*\wedge+i_U
 ,
\end{align}
where $U^*\wedge$ and $i_U$ are the exterior and interior
multiplications by $U^*$ and $U$ on $\Lambda^*(F_i^{\perp })$.

Denote  $q={\rm rk}( F)$,  $q_i={\rm rk}(F_i^\perp)$.

Let $h_1,\, \cdots,\, h_{q_i}$ be an oriented orthonormal basis of
$F_i^\perp$. Set
\begin{align}\label{1.33}
\tau\left(F_i^\perp,g^{F_i^\perp}\right)=\left(\frac{1}{\sqrt{-1}}\right)^{\frac{q_i(q_i+1)}{2}}c\left(h_1\right)\cdots
c\left(h_{q_i}\right).
\end{align}
Then
\begin{align}\label{1.34}
\tau\left(F_i^\perp,g^{F_i^\perp}\right)^2={\rm
Id}_{\Lambda^*\left(F_i^{\perp }\right)}.
\end{align}
Set
\begin{align}\label{1.35}
 {\Lambda^*_\pm\left(F_i^{\perp }\right)}=
 \left\{ h\in \Lambda^*\left(F_i^{\perp }\right):\tau\left(F_i^{\perp },g^{F_i^\perp}\right)h=\pm h\right\} .
\end{align}

Since $q_i$ is even, for any $h\in  F_i^\perp$, $c(h)$
anti-commutes with $\tau (F_i^\perp,g^{F_i^\perp} )$, while
$\widehat{c}(h)$ commutes with $\tau (F_i^\perp,g^{F_i^\perp} )$.
In particular, $c(h)$ exchanges $ {\Lambda^*_\pm (F_i^{\perp }
)}$.

Let $\widetilde{\tau}(F_i^\perp  )$ denote the ${\bf Z}_2$-grading
of $ {\Lambda ^* (F_i^{\perp } )}$ defined by
\begin{align}\label{1.351}
 \left. \widetilde{\tau}\left(F_i^\perp
 \right)\right|_{\Lambda^{\frac{\rm even}{\rm odd}} \left(F_i^{\perp }
 \right)}=\pm{\rm Id}|_{\Lambda^{\frac{\rm even}{\rm odd}} \left(F_i^{\perp }
 \right)}.
\end{align}

Now we have  the following  ${\bf Z}_2$-graded vector bundles over
$M$:
\begin{align}\label{1.36}
S(F)=S_+(F)\oplus S_-(F),
\end{align}
\begin{align}\label{1.37}
\Lambda ^*\left(F_i^{\perp }\right)=\Lambda^*_+\left(F_i^{\perp
}\right)\oplus\Lambda^*_-\left(F_i^{\perp }\right),\ \ \ i=1,\ 2,
\end{align}
and
\begin{align}\label{1.37a}
\Lambda^* \left(F_i^{\perp }\right)=\Lambda^{\rm
even}\left(F_i^{\perp }\right)\oplus\Lambda^{\rm
odd}\left(F_i^{\perp }\right),\ \ \ i=1,\ 2.
\end{align}

 We form the following  ${\bf Z}_2$-graded  tensor product, which
 will play a role in Section 2:
\begin{align}\label{1.38}
 W\left(F,F^\perp_1,F^\perp_2\right)=S(F)\widehat{\otimes}
 \Lambda^*
\left(F_1^\perp\right)\widehat{\otimes}\Lambda^*
\left(F_2^\perp\right) ,
\end{align}
with the ${\bf Z}_2$-grading operator given by
\begin{align}\label{1.39}
\tau_W=\tau_{S(F)}\cdot
\tau\left(F_1^\perp,g^{F_1^\perp}\right)\cdot
\widetilde{\tau}\left(F_2^\perp \right),
\end{align}
where $\tau_{S(F)}$ is the ${\bf Z}_2$-grading operator defining the
splitting in (\ref{1.36}). We denote by
\begin{align}\label{1.40}
 W\left(F,F^\perp_1,F^\perp_2\right)=W_+\left(F,F^\perp_1,F^\perp_2\right)\oplus W_-\left(F,F^\perp_1,F^\perp_2\right)
\end{align}
the ${\bf Z}_2$-graded decomposition with respect to $\tau_W$.

Recall that the connections $\nabla^F$, $\nabla^{F_1^\perp}$ and
$\nabla^{F_2^\perp}$ have been defined in (\ref{1.9a}).  They lift canonically to Hermitian
connections $\nabla^{S(F)}$, $\nabla^{ \Lambda^*
 (F_1^\perp )}$, $\nabla^{ \Lambda^*
(F_2^\perp)}$ on $S(F)$, $\Lambda^*
(F_1^\perp)$, $\Lambda^* (F_2^\perp)$
respectively, preserving the corresponding ${\bf Z}_2$-gradings.
Let $\nabla^{W(F,F_1^\perp,F_2^\perp)}$ be the canonically induced
connection on $W(F,F_1^\perp,F_2^\perp)$ which preserves the
canonically induced Hermitian metric on
$W(F,F_1^\perp,F_2^\perp)$, and also the ${\bf Z}_2$-grading of
$W(F,F_1^\perp,F_2^\perp)$.

For any vector bundle $E$ over $M$, by an integral polynomial of $E$
we will mean a bundle $\phi(E)$ which is a polynomial in the
exterior and symmetric powers of $E$ with integral coefficients.

For $i=1,\, 2$, let $\phi_i(F_i^\perp)$ be an integral polynomial of
$F_i^\perp$. We   denote the complexification of $\phi_i(F_i^\perp)$
by the same notation. Then $\phi_i(F_i^\perp)$ carries a naturally
induced Hermitian metric from $g^{F_i^\perp}$ and also a naturally
induced Hermitian connection $\nabla^{\phi_i(F_i^\perp)}$ from
$\nabla^{F_i^\perp}$.

Let $W(F,F_1^\perp,F_2^\perp)\otimes \phi_1(F_1^\perp)\otimes
\phi_2(F_2^\perp)$ be the ${\bf Z}_2$-graded vector bundle over $M$,
\begin{multline}\label{1.41}
 W\left(F,F_1^\perp,F_2^\perp\right)\otimes \phi_1\left(F_1^\perp\right)\otimes
\phi_2\left(F_2^\perp\right)=W_+\left(F,F_1^\perp,F_2^\perp\right)\otimes
\phi_1\left(F_1^\perp\right)\otimes \phi_2\left(F_2^\perp\right)\\
\oplus W_-\left(F,F_1^\perp,F_2^\perp\right)\otimes
\phi_1\left(F_1^\perp\right)\otimes \phi_2\left(F_2^\perp\right).
\end{multline}
Let $\nabla^{W\otimes\phi_1\otimes\phi_2}$ denote the naturally
induced Hermitian connection on $W(F,F_1^\perp,F_2^\perp)\otimes
\phi_1(F_1^\perp)\otimes \phi_2(F_2^\perp)$ with respect to the
naturally induced Hermitian metric on it. Clearly,
$\nabla^{W\otimes\phi_1\otimes\phi_2}$ preserves the ${\bf
Z}_2$-graded decomposition  in (\ref{1.41}).

Let $S$ be the ${\rm End}(TM)$-valued one form on $M$ defined by
\begin{align}\label{1.42}
 \nabla^{TM}=\nabla^F+\nabla^{F_1^\perp}+\nabla^{F_2^\perp}+S.
\end{align}
Let $e_1,\,\cdots,\,e_{\dim M}$ be an orthonormal basis of $TM$. Let
$\nabla ^{F,\phi_1(F^\perp_1)\otimes \phi_2(F^\perp_2)}$ be the
Hermitian connection on $W (F,F_1^\perp,F_2^\perp )\otimes
\phi_1(F_1^\perp)\otimes \phi_2(F_2^\perp)$ defined by that for any
$X\in\Gamma(TM)$,
\begin{align}\label{1.42a}
 \nabla ^{F,\phi_1(F^\perp_1)\otimes \phi_2(F^\perp_2)}_X= \nabla_{X}^{W\otimes\phi_1\otimes\phi_2}+
 \frac{1}{4}\sum_{i,\,j=1}^{\dim M}\left\langle S(X)e_i,e_j\right\rangle  c\left(e_i\right)c\left(e_j\right).
\end{align}

Let the linear operator $D^{F,\phi_1(F^\perp_1)\otimes
\phi_2(F^\perp_2)}:\Gamma(W (F,F_1^\perp,F_2^\perp )\otimes \phi_1
(F_1^\perp )\otimes \phi_2 (F_2^\perp ))\rightarrow \Gamma(W
(F,F_1^\perp,F_2^\perp )\otimes \phi_1 (F_1^\perp )\otimes \phi_2
(F_2^\perp ))$ be    defined by 
\begin{align}\label{1.43}
D^{F,\phi_1(F^\perp_1)\otimes \phi_2(F^\perp_2)}=\sum_{i=1}^{\dim M}
c\left(e_i\right)\nabla_{e_i}^{F, \phi_1(F_1^\perp)\otimes
\phi_2(F_2^\perp)}.
\end{align}
We call $D^{F,\phi_1(F^\perp_1)\otimes \phi_2(F^\perp_2)}$ a
sub-Dirac operator with respect to the spin vector bundle $F$.

One verifies  that $D^{F,\phi_1(F^\perp_1)\otimes
\phi_2(F^\perp_2)}$ is a first order formally self-adjoint elliptic
differential operator. Moreover, it exchanges $\Gamma(W_\pm
(F,F_1^\perp,F_2^\perp )\otimes \phi_1 (F_1^\perp )\otimes \phi_2
(F_2^\perp ))$. We denote by $D_\pm^{F,\phi_1(F^\perp_1)\otimes
\phi_2(F^\perp_2)}$ the restrictions of
$D^{F,\phi_1(F^\perp_1)\otimes \phi_2(F^\perp_2)}$ to $\Gamma(W_\pm
(F,F_1^\perp,F_2^\perp )\otimes \phi_1 (F_1^\perp )\otimes \phi_2
(F_2^\perp ))$. 
Then one has
\begin{align}\label{1.44}
 \left(D_+^{F,\phi_1(F^\perp_1)\otimes \phi_2(F^\perp_2)}\right)^*=D_-^{F,\phi_1(F^\perp_1)\otimes \phi_2(F^\perp_2)} .
\end{align}

\begin{rem}\label{t1.40}  In the special case of $F=\{0\}$, the above sub-Dirac operator is simply the sub-Signature operator constructed in \cite{Z96} (cf. \cite{Z04}). On the other hand, in the case of $F_2^\perp=\{0\}$, the above sub-Dirac operator is constructed in \cite[Section 2]{LZ01}, which is sufficient for the proof of Theorem \ref{t0.2}. The sub-Dirac operator constructed above will be used in Section \ref{s2.6} to prove the Connes vanishing theorem, i.e., Theorem \ref{t0.1}.
\end{rem}

\begin{rem}\label{t1.4} When
$F_1^\perp$, $F_2^\perp$ are also spin and carry fixed spin
structures, then    $TM=F\oplus F^\perp_1\oplus F^\perp_2$ is spin
and carries an induced spin structure from the spin structures on
$F$, $F_1^\perp$ and $F_2^\perp$. Moreover, one has the following
identifications  of ${\bf Z}_2$-graded vector bundles (cf.
\cite{LaMi89}) for $i=1,\, 2$,
\begin{align}\label{1.44a}
  \Lambda^*_+\left(F_i^\perp\right)\oplus \Lambda^*_-\left(F_i^\perp\right)=S_+ \left(F_i^\perp\right)\otimes
  S
  \left(F_i^\perp\right)^*\oplus S_- \left(F_i^\perp\right)\otimes
  S
  \left(F_i^\perp\right)^*,
\end{align}
\begin{multline}\label{1.44b}
  \Lambda^{\rm even}\left(F_i^\perp\right)\oplus \Lambda^{\rm odd}
  \left(F_i^\perp\right)=\left(S_+ \left(F_i^\perp\right)\otimes
  S_+
  \left(F_i^\perp\right)^*\oplus S_- \left(F_i^\perp\right)\otimes
  S_-
  \left(F_i^\perp\right)^*\right)
\\ \oplus   \left(S_+ \left(F_i^\perp\right)\otimes
  S_-
  \left(F_i^\perp\right)^*\oplus S_- \left(F_i^\perp\right)\otimes
  S_+
  \left(F_i^\perp\right)^*\right). \end{multline}
By (\ref{1.33})-(\ref{1.43}), (\ref{1.44a}) and (\ref{1.44b}),
$D^{F,\phi_1(F^\perp_1)\otimes \phi_2(F^\perp_2)}$ is simply the
twisted   Dirac operator
\begin{multline}\label{1.45}
D^{F,\phi_1(F^\perp_1)\otimes \phi_2(F^\perp_2)}:
\Gamma\left(S(TM)\widehat{\otimes} S\left(F_2^\perp\right)^*\otimes
S\left(F_1^\perp\right)^*\otimes\phi_1\left(F^\perp_1\right)\otimes
\phi_2\left(F^\perp_2\right)
 \right)\\
 \longrightarrow \Gamma\left(S(TM)\widehat{\otimes }
S\left(F_2^\perp\right)^*\otimes
S\left(F_1^\perp\right)^*\otimes\phi_1\left(F^\perp_1\right)\otimes
\phi_2\left(F^\perp_2\right)
 \right),
\end{multline}
where for $i=1,\, 2$, the Hermitian (dual) bundle of spinors
$S(F_i^\perp)^*$  associated to $(F_i^\perp,g^{F_i^\perp})$ carries
the Hermitian connection induced from $\nabla^{F_i^\perp}$.
The point of (\ref{1.43}) is that it only requires $F$ being spin.
While on the other hand, (\ref{1.45}) allows us to take the
advantage of applying the calculations already done for usual
(twisted) Dirac operators when doing local computations.
\end{rem}

\begin{rem}\label{t1.4a} It is clear that the definition in
(\ref{1.43}) does not require that $F\subseteq TM$ being   integrable.  
\end{rem}

Let $\Delta^{F,\phi_1(F^\perp_1)\otimes \phi_2(F^\perp_2)}$ denote
the Bochner Laplacian defined by
\begin{align}\label{1.46}
\Delta^{F,\phi_1(F^\perp_1)\otimes
\phi_2(F^\perp_2)}=\sum_{i=1}^{\dim M}
\left(\nabla_{e_i}^{F,\phi_1(F^\perp_1)\otimes
\phi_2(F^\perp_2)}\right)^2-\nabla^{F,\phi_1(F^\perp_1)\otimes
\phi_2(F^\perp_2)}_{\sum_{i=1}^{\dim M}\nabla^{TM}_{e_i}e_i}.
  \end{align}

 Let
$k^{TM}$ be the scalar curvature of $g^{TM}$, $R^{F_i^\perp}$
($i=1,\, 2$) be the curvature of $\nabla^{F_i^\perp}$. Let
$R^{\phi_1(F_1^\perp)\otimes \phi_2(F_2^\perp)}$  be the curvature of the tensor product connection on $\phi_1(F_1^\perp)\otimes \phi_2(F_2^\perp)$ induced from
$\nabla^{\phi_1(F_1^\perp)}$ and $\nabla^{\phi_2(F_2^\perp)}$.

 In view of Remark \ref{t1.4}, the
following Lichnerowicz type formula holds:
\begin{multline}\label{1.47}
\left(D^{F,\phi_1(F^\perp_1)\otimes
\phi_2(F^\perp_2)}\right)^2=-\Delta^{F,\phi_1(F^\perp_1)\otimes
\phi_2(F^\perp_2)}+\frac{k^{TM}}{4}
+
\frac{1}{2}\sum_{i,\,j=1}^{\dim M}c\left(e_i\right)c\left(e_j\right)R^{ \phi_1(F^\perp_1)\otimes
\phi_2(F^\perp_2)}\left(e_i,e_j\right)
\\
+\frac{1}{8}\sum_{k=1}^2\sum_{i,\, j,\,s,\,t=1}^{\dim M} \left\langle
R^{F_k^\perp}\left(e_i,e_j\right)e_t,e_s\right\rangle
c\left(e_i\right)c\left(e_j\right)\widehat{c}\left(e_s\right)\widehat{c}\left(e_t\right) .
\end{multline}

When $M$ is compact,   by   the Atiyah-Singer index theorem
\cite{ASI} (cf. \cite{LaMi89}), one has
\begin{multline}\label{1.48}
{\rm ind}\left( D_+^{F,\phi_1(F^\perp_1)\otimes
\phi_2(F^\perp_2)}\right)\\ = 2^{\frac{q_1}{2} }\left\langle
\widehat{A}(F)\widehat{L}\left(F_1^\perp\right) 
e\left(F_2^\perp\right){\rm
ch}\left(\phi_1\left(F_1^\perp\right)\right){\rm
ch}\left(\phi_2\left(F_2^\perp\right)\right),[M]\right\rangle,
\end{multline}
where  $\widehat A(F)$ is the   Hirzebruch $\widehat A$-class (cf. \cite[\S 1.6.3]{Z01}) of $F$, $\widehat{L}(F_1^\perp)$   is the Hirzebruch
$\widehat{L}$-class (cf. \cite[(11.18') of Chap. III]{LaMi89}) of
$F_1^\perp$, $e(F_2^\perp)$ is the Euler class (cf. \cite[\S
3.4]{Z01}) of $F_2^\perp$, and ``${\rm ch}$'' is the notation for
the Chern character (cf. \cite[\S1.6.4]{Z01}).

\subsection{A vanishing theorem for   almost isometric
foliations}\label{s1.5} In this subsection, we assume $M$ is compact
and prove a vanishing theorem.  Some of the computations  in
this subsection will be used in the next section where we will deal
with the case where $M$ is non-compact.

Let $f_1,\, \cdots,\, f_q$ be an oriented orthonormal basis of $F$.
Let $h_1,\,\cdots,\,h_{q_1}$ (resp. $e_1,\,\cdots,\,e_{q_2}$) be an
oriented orthonormal basis of $F_1^\perp$ (resp. $F_2^\perp$).

Let $\beta>0$, $\varepsilon>0$ and   consider the construction in Section
\ref{s1.4} with respect to the metric $g^{TM}_{\beta,\varepsilon}$ defined
in (\ref{1.8}). We still use the superscripts ``$\beta$,
$\varepsilon$'' to decorate the geometric data associated to
$g^{TM}_{\beta,\varepsilon} $. For example, $D^{F,\phi_1(F^\perp_1)\otimes
\phi_2(F^\perp_2),\beta,\varepsilon }$ now denotes the sub-Dirac
operator constructed in (\ref{1.43}) associated  to
$g^{TM}_{\beta,\varepsilon} $. Moreover, it can be written as
\begin{multline}\label{1.431}
D^{F,\phi_1(F^\perp_1)\otimes
\phi_2(F^\perp_2),\beta,\varepsilon}=\beta^{-1}\sum_{i=1}^{q}
c\left(f_i\right)\nabla_{f_i}^{F, \phi_1(F_1^\perp)\otimes
\phi_2(F_2^\perp),\beta,\varepsilon}
+\varepsilon \sum_{j=1}^{q_1}
c\left(h_j\right)\nabla_{h_j}^{F, \phi_1(F_1^\perp)\otimes
\phi_2(F_2^\perp),\beta,\varepsilon}\\ + \sum_{s=1}^{q_2}
c\left(e_s\right)\nabla_{e_s}^{F, \phi_1(F_1^\perp)\otimes
\phi_2(F_2^\perp),\beta,\varepsilon}.
\end{multline}

By (\ref{1.431}), the Lichnerowicz type formula (\ref{1.47}) for
$(D^{F,\phi_1(F^\perp_1)\otimes \phi_2(F^\perp_2),\beta,\varepsilon })^2$
takes the following form (compare with
\cite[Theorem 2.3]{LZ01}),
\begin{multline}\label{1.49}
\left(D^{F,\phi_1(F^\perp_1)\otimes
\phi_2(F^\perp_2),\beta,\varepsilon}\right)^2=-\Delta^{F,\phi_1(F^\perp_1)\otimes
\phi_2(F^\perp_2),\beta,\varepsilon}+\frac{k^{TM,\beta,\varepsilon}}{4}
\\
+
\frac{1}{2\beta^2}\sum_{i,\,j=1}^qc\left(f_i\right)c\left(f_j\right)
R^{\phi_1(F^\perp_1)\otimes \phi_2(F^\perp_2),\beta,\varepsilon}\left(f_i,f_j\right)
+\frac{\varepsilon^2 }{2}\sum_{i,\,j=1}^{q_1}c\left(h_i\right)
c\left(h_j\right)R^{\phi_1(F^\perp_1)\otimes \phi_2(F^\perp_2),\beta,\varepsilon}\left(h_i,h_j\right) 
\end{multline}
$$
+\frac{1}{2}\sum_{i,\,j=1}^{q_2}c\left(e_i\right)c\left(e_j\right)R^{\phi_1(F^\perp_1)\otimes \phi_2(F^\perp_2),\beta,\varepsilon}\left(e_i,e_j\right)
+\frac{\varepsilon }{\beta}\sum_{i =1}^q\sum_{j
=1}^{q_1}c\left(f_i\right)c\left(h_j\right)R^{\phi_1(F^\perp_1)\otimes \phi_2(F^\perp_2),\beta,\varepsilon}\left(f_i,h_j\right)$$
$$
 +\frac{1}{\beta}\sum_{i =1}^q\sum_{j
=1}^{q_2}c\left(f_i\right)c\left(e_j\right)R^{\phi_1(F^\perp_1)\otimes \phi_2(F^\perp_2),\beta,\varepsilon}\left(f_i,e_j\right)
+\varepsilon \sum_{i =1}^{q_1}\sum_{j
=1}^{q_2}c\left(h_i\right)c\left(e_j\right)R^{\phi_1(F^\perp_1)\otimes \phi_2(F^\perp_2),\beta,\varepsilon}\left(h_i,e_j\right)
 $$
 $$
+\frac{1}{8\beta^2}\sum_{i,\, j=1}^q\sum_{s,\, t=1}^{q_1}\left\langle
R^{F_1^\perp,\beta,\varepsilon}\left(f_i,f_j\right)h_t,h_s\right\rangle
c\left(f_i\right)c\left(f_j\right)\widehat{c}\left(h_s\right)\widehat{c}\left(h_t\right)
$$
$$ +\frac{\varepsilon^2}{8}\sum_{i,\, j=1}^{q_1}\sum_{s,\,
t=1}^{q_1}\left\langle
R^{F_1^\perp,\beta,\varepsilon}\left(h_i,h_j\right)h_t,h_s\right\rangle
c\left(h_i\right)c\left(h_j\right)\widehat{c}\left(h_s\right)\widehat{c}\left(h_t\right)$$
$$
+\frac{1}{8}\sum_{i,\, j=1}^{q_2}\sum_{s,\,
t=1}^{q_1}\left\langle
R^{F_1^\perp,\beta,\varepsilon}\left(e_i,e_j\right)h_t,h_s\right\rangle
c\left(e_i\right)c\left(e_j\right)\widehat{c}\left(h_s\right)\widehat{c}\left(h_t\right)$$
$$
+\frac{\varepsilon }{4\beta}\sum_{i =1}^q\sum_{j=1
}^{q_1}\sum_{s,\, t=1}^{q_1}\left\langle
R^{F_1^\perp,\beta,\varepsilon}\left(f_i,h_j\right)h_t,h_s\right\rangle
c\left(f_i\right)c\left(h_j\right)\widehat{c}\left(h_s\right)\widehat{c}\left(h_t\right)$$
$$
+\frac{1}{4\beta}\sum_{i =1}^q\sum_{j=1 }^{q_2}\sum_{s,\,
t=1}^{q_1}\left\langle
R^{F_1^\perp,\beta,\varepsilon}\left(f_i,e_j\right)h_t,h_s\right\rangle
c\left(f_i\right)c\left(e_j\right)\widehat{c}\left(h_s\right)\widehat{c}\left(h_t\right)$$
$$
+\frac{\varepsilon }{4}\sum_{i=1
}^{q_1}\sum_{j=1 }^{q_2}\sum_{s,\, t=1}^{q_1}\left\langle
R^{F_1^\perp,\beta,\varepsilon}\left(h_i,e_j\right)h_t,h_s\right\rangle
c\left(h_i\right)c\left(e_j\right)\widehat{c}\left(h_s\right)\widehat{c}\left(h_t\right)
$$
$$
+\frac{1}{8\beta^2}\sum_{i,\, j=1}^q\sum_{s,\, t=1}^{q_2}\left\langle
R^{F_2^\perp,\beta,\varepsilon}\left(f_i,f_j\right)e_t,e_s\right\rangle
c\left(f_i\right)c\left(f_j\right)\widehat{c}\left(e_s\right)\widehat{c}\left(e_t\right)
$$
$$ +\frac{\varepsilon^{2}}{8}\sum_{i,\, j=1}^{q_1}\sum_{s,\,
t=1}^{q_2}\left\langle
R^{F_2^\perp,\beta,\varepsilon}\left(h_i,h_j\right)e_t,e_s\right\rangle
c\left(h_i\right)c\left(h_j\right)\widehat{c}\left(e_s\right)\widehat{c}\left(e_t\right)$$
$$
+\frac{1 }{8}\sum_{i,\, j=1}^{q_2}\sum_{s,\,
t=1}^{q_2}\left\langle
R^{F_2^\perp,\beta,\varepsilon}\left(e_i,e_j\right)e_t,e_s\right\rangle
c\left(e_i\right)c\left(e_j\right)\widehat{c}\left(e_s\right)\widehat{c}\left(e_t\right)$$
$$
+\frac{\varepsilon }{4\beta}\sum_{i=1 }^q\sum_{j=1
}^{q_1}\sum_{s,\, t=1}^{q_2}\left\langle
R^{F_2^\perp,\beta,\varepsilon}\left(f_i,h_j\right)e_t,e_s\right\rangle
c\left(f_i\right)c\left(h_j\right)\widehat{c}\left(e_s\right)\widehat{c}\left(e_t\right)$$
$$
+\frac{1}{4\beta}\sum_{i =1}^q\sum_{j=1 }^{q_2}\sum_{s,\,
t=1}^{q_2}\left\langle
R^{F_2^\perp,\beta,\varepsilon}\left(f_i,e_j\right)e_t,e_s\right\rangle
c\left(f_i\right)c\left(e_j\right)\widehat{c}\left(e_s\right)\widehat{c}\left(e_t\right)$$
$$
+\frac{\varepsilon }{4}\sum_{i
=1}^{q_1}\sum_{j=1 }^{q_2}\sum_{s,\, t=1}^{q_2}\left\langle
R^{F_2^\perp,\beta,\varepsilon}\left(h_i,e_j\right)e_t,e_s\right\rangle
c\left(h_i\right)c\left(e_j\right)\widehat{c}\left(e_s\right)\widehat{c}\left(e_t\right).
$$

 By (\ref{1.24}), (\ref{1.31}), (\ref{1.101}) and (\ref{1.49}), we get that when $\beta>0$,
 $\varepsilon>0$ are small,
 \begin{align}\label{1.50}
\left(D^{F,\phi_1(F^\perp_1)\otimes
\phi_2(F^\perp_2),\beta,\varepsilon}\right)^2=-\Delta^{F,\phi_1(F^\perp_1)\otimes
\phi_2(F^\perp_2),\beta,\varepsilon}+\frac{k^{F}}{4\beta^2}
+O\left( \frac{1 }{\beta}+\frac{\varepsilon^2 }{\beta^2}\right).
 \end{align}

 \begin{prop}\label{t1.6} If $k^F>0$ over $M$, then for any Pontrjagin classes
 $p(F_1^\perp)$,
 $p'(F_2^\perp)$  of $F_1^\perp$,
 $F_2^\perp$ respectively, the following identity holds,
 \begin{align}\label{1.51}
\left\langle\widehat{A}(F)p\left(F_1^\perp\right)e\left(F_2^\perp\right)p'\left(F_2^\perp\right),[M]\right\rangle
=0.
 \end{align}
 \end{prop}
 \begin{proof}Since $k^F>0$ over $M$, one can take $\beta>0$, $\varepsilon>0$
 small enough so that the corresponding terms in the right hand side of
 (\ref{1.50}) verifies that
 \begin{align}\label{1.52}
 \frac{k^{F}}{4\beta^2} +O\left( \frac{1 }{\beta}+\frac{\varepsilon^2 }{\beta^2}\right)>0
 \end{align}
 over $M$.
Since  $-\Delta^{F,\phi_1(F^\perp_1)\otimes
\phi_2(F^\perp_2),\beta,\varepsilon}$ is   nonnegative, by
(\ref{1.44}), (\ref{1.50}) and (\ref{1.52}), one gets
\begin{align}\label{1.53}
{\rm ind}\left( D_+^{F,\phi_1(F^\perp_1)\otimes
\phi_2(F^\perp_2),\beta,\varepsilon}\right)= 0.
\end{align}

From (\ref{1.48}) and (\ref{1.53}), we get
\begin{align}\label{1.54}\left\langle
\widehat{A}(F)\widehat{L}\left(F_1^\perp\right){\rm
ch}\left(\phi_1\left(F_1^\perp\right)\right) 
e\left(F_2^\perp\right){\rm
ch}\left(\phi_2\left(F_2^\perp\right)\right),[M]\right\rangle=0.
\end{align}

Now as it is standard that any rational Pontrjagin class of $F_1^\perp$
(resp. $F_2^\perp$)  can be expressed as a rational linear
combination of   classes of the form $\widehat{L} (F_1^\perp
){\rm ch} (\phi_1 (F_1^\perp ) )$ (resp. $ {\rm ch} (\phi_2 (F_2^\perp ) )$), one gets (\ref{1.51}) from
(\ref{1.54}).
 \end{proof}

 \begin{rem}\label{t1.7}  If one changes the ${\bf Z}_2$-grading in the definition of the sub-Dirac operator by replacing
 $\widetilde{\tau}(F_2^\perp)$ in (\ref{1.39}) by $\tau(F_2^\perp,
 g^{F_2^\perp})$, then one can prove that under the same condition
 as in Proposition \ref{t1.6},  
\begin{align}\label{1.54a}\left\langle
\widehat{A}(F)
p\left(F_1^\perp\right)p'\left(F_2^\perp\right),[M]\right\rangle=0
\end{align}
for any Pontrjagin classes $p(F_1^\perp)$, $p'(F_2^\perp)$ of
$F_1^\perp$, $F_2^\perp$.
 \end{rem}

 %%%%%%%%%%%%%%%%%%%%%%%%%%%%%%%%%%%%%%%%%%%
%%%%%%%%%%%%%%%%%%%%%%%%%%%%%%%%%%%%%%%%%%%
\section{Connes fibration and vanishing theorems}\label{s2}
%%%%%%%%%%%%%%%%%%%%%%%%%%%%%%%%%%%%%%%%%%%

 This section is organized as follows. In Section \ref{s2.1}, we
recall the definition of the Connes fibration and prove some basic properties of it.  
In Section \ref{s2.3}, we introduce a specific deformation of the     sub-Dirac operator on the Connes fibration and prove a key  vanishing result for the deformed sub-Dirac operator on certain compact subsets of the Connes fibration. This motivates the proof of Theorem \ref{t0.2} for the case of $\dim M=4k$ given in    Section \ref{s2.4}. In Section \ref{s2.5}, we present the proof of the $\dim M=8k+i$ ($i=1,\ 2$) cases of Theorem \ref{t0.2}. Finally, in Section \ref{s2.6} we present the proof of Theorem \ref{t0.1}, and   state some new vanishing results.

%%%%%%%%%%%%%%%%%%%%%%%%%%%%%%%%%%%%%%%%%%%
\subsection{The Connes fibration
}\label{s2.1} 

Let $(M,F)$ be a compact foliation, i.e., $F$ is an
integrable subbundle of the tangent vector bundle $TM$ of a closed
manifold $M$.  
For any  vector space $E$ of rank $n$, let $\mE$ be the
set of all Euclidean metrics on $E$. It is well known that $\mE$
is the noncompact homogeneous space $GL(n,{\bf R})^+/SO(n)$ (with $\dim
\mE=\frac{n(n+1)}{2}$), which carries a
natural Riemannian metric of nonpositive sectional curvature (cf. \cite{He}). In particular, any two points of $\mE$ can be joined
by a unique geodesic.

Following  \cite[\S 5]{Co86}, let $\pi:\mM\rightarrow M$ be the
fibration over $M$ such that for any $x\in M$, $\mM_x=\pi^{-1}(x)$
is the space of Euclidean metrics on the linear space $T_xM/F_x$.

Let  $T^V\mM$ denote the vertical tangent bundle of the fibration
$\pi:\mM\rightarrow M$. Then it carries a natural metric
$g^{T^V\mM}$   such that   any two points $p,\,
q\in\mM_x$, with $x\in M$,  can be joined by a unique geodesic in
$\mM_x$. Let $d^{\mM_x}(p,q)$ denote the length of this geodesic.

By  using the Bott connection
  on $TM/F$ (cf. (\ref{1.0a})), which is leafwise flat, one  lifts $F$ to an integrable subbundle
$\mF$ of $T\mM$.{\footnote{Indeed,
the Bott connection on $TM/F$ determines an integrable lift
$\widetilde{\mF}$ of $F$ in $T\widetilde{\mM}$, where
$\widetilde{\mM}=GL(TM/F)^+$ is the $GL(q_1,{\bf R})^+$ (with
$q_1={\rm rk}(TM/F)$) principal bundle of oriented frames  over
$M$. Now as $\widetilde{\mM}$ is a principal $SO(q_1)$ bundle over
$\mM$, $\widetilde{\mF}$ determines an integrable subbundle $\mF$
of $T\mM$.}  }
  Let $g^F$ be a Euclidean metric on $F$, which  lifts to a Euclidean metric $g^\mF=\pi^*g^F $ on $\mF$.

For any $v\in\mM$, $T_v\mM/(\mF_v\oplus T^V_v\mM)$ is identified with
$T_{\pi(v)}M/F_{\pi(v)}$ under the projection $\pi:\mM\rightarrow
M$. By definition, $v$ determines a metric on
$T_{\pi(v)}M/F_{\pi(v)}$, which in turn determines a metric on
$T_v\mM/(\mF_v\oplus T^V_v\mM)$. In this way, $T\mM/(\mF \oplus T^V
\mM)$ carries a canonically induced metric.

Let $\mF_1^\perp\subseteq T\mM$ be a subbundle, which is  transversal to $\mF\oplus T^V\mM$,   such that we have a
splitting $T\mM=(\mF \oplus T^V \mM)\oplus\mF_1^\perp$. Then
$\mF_1^\perp$ can be identified with $T\mM/(\mF \oplus T^V \mM)$
and carries a canonically induced metric $g^{\mF_1^\perp}$.
We denote from now on that $\mF_2^\perp=T^V\mM$.

Let $g^{T\mM}$ be the Riemannian metric on $\mM$ defined by the
following orthogonal splitting,
\begin{align}\label{2.1a}\begin{split}
       T\mM =   \mF\oplus \mF^\perp_1\oplus \mF^\perp_2,\ \ \ \ \ \
g^{T\mM }=   g^{\mF}\oplus g^{\mF^\perp_1}\oplus
g^{\mF^\perp_2}.\end{split}
\end{align}
Let   $p_2^\perp$ be the orthogonal projection from $T\mM$ to  $\mF_2^\perp$.  Let $\nabla^{T\mM}$ be the Levi-Civita connection of $g^{T\mM}$. Then $\nabla^{\mF_2^\perp}=p_2^\perp \nabla^{T\mM}p_2^\perp$ is a Euclidean connection on $\mF_2^\perp$ not depending on $g^\mF$ and $g^{\mF_1^\perp}$.

 By   \cite[Lemma 5.2]{Co86}, $(\mM,\mF)$
admits  an almost isometric structure with respect to the metrics given by
(\ref{2.1a}). In particular, for any $X\in \Gamma(\mF)$,
$U_i,\ V_i\in\Gamma(\mF_i^\perp)$ with $i=1,\ 2$, one has by  (\ref{1.5}) that                
\begin{align}\label{1.5x} \begin{split}
       \langle [X,U_i],V_i\rangle+\langle U_i,[X,V_i]\rangle =X\langle U_i,V_i\rangle,\\
\langle [X,U_2],U_1\rangle =0.\end{split}
\end{align}

Take a metric on ${TM/F}$. This is
equivalent to taking an embedded section $s:M\hookrightarrow \mM$
of the Connes fibration $\pi:\mM\rightarrow M$. Then we have a
canonical inclusion  $s(M)\subset \mM$.  

For any $p\in \mM\setminus s(M)$,  we connect $p$ and $s(\pi(p))\in s(M)$ by the unique geodesic in $\mM_{\pi(p)}$. 
Let $\sigma(p)\in\mF_{2}^\perp|_p$ denote the unit vector tangent to this geodesic.  
Let $\rho(p)=d^{\mM_{\pi(p)}}(p,s(\pi(p)))$ denote the length of this geodesic. 

The following simple result  will play a  key role in what follows.

\begin{lemma}\label{t2.1}
 There exists   $A_1>0$, depending only on   the embedding $s:M\hookrightarrow \mM$, such that for any $X\in\Gamma(\mF)$ with $|X|\leq 1$, the following pointwise inequalities hold on $\mM\setminus s(M)$,
\begin{align}\label{2.1}
|X(\rho)|\leq A_1,
\end{align}
\begin{align}\label{2.2}
\left| \nabla^{\mF_2^\perp}_X\sigma\right|\leq \frac{A_1}{\rho}.
\end{align}
In particular, the following inequality holds on $\mM$, 
\begin{align}\label{2.2a}
\left| \nabla^{\mF_2^\perp}_X(\rho\sigma)\right|\leq  2{A_1} .
\end{align}
\end{lemma}

\begin{proof} Since the estimates to be proved are local, we may well assume that there is $Y\in\Gamma(F)$ over  $M$, with $|Y|\leq 1$,  such that $X=\pi^*Y$.
Let $\phi_t$ (resp. $\widetilde\phi_t$), $t\in{\bf R}$, be the one-parameter group of diffeomorphisms on $M$ (resp. $\mM$) generated by $Y$ (resp. $X=\pi^*Y$).  Then $\widetilde\phi_t$ is the lift of $\phi_t$. 

Take any $p\in\mM\setminus s(M)$. By \cite[Lemma 5.2]{Co86} and (\ref{1.5x}), one sees that each  $\widetilde\phi_t$ maps 
the fiber $\mM_{\pi(p)}$ isometrically to the fiber $\mM_{\phi_t(\pi(p))}$. Thus, it maps
the geodesic connecting $p$ and $s(\pi(p))$ in $\mM_{\pi(p)}$ to the geodesic connecting $\widetilde\phi_t(p)$ and $\widetilde\phi_t(s(\pi(p)))$ in $\mM_{\phi_t(\pi(p))}$, such that  
$
 \rho(p)=d^{\mM_{\phi_t(\pi(p))}}  ( \widetilde\phi_t(p),\widetilde\phi_t(s(\pi(p))) ) .
$
 Thus, one has
 \begin{multline}\label{2.7}
\left|\rho\left(\widetilde\phi_t(p)\right)- \rho(p)\right|=\left|d^{\mM_{\phi_t(\pi(p))}} \left( \widetilde\phi_t(p), s(\phi_t(\pi(p)))\right)
-   d^{\mM_{\phi_t(\pi(p))}} \left( \widetilde\phi_t(p),\widetilde\phi_t(s(\pi(p)))\right) \right|
\\
\leq d^{\mM_{\phi_t(\pi(p))}} \left(s(\phi_t(\pi(p))),\widetilde\phi_t(s(\pi(p)))\right) 
=\rho\left(\widetilde\phi_t(s(\pi(p)))\right). 
\end{multline}

Since at $p$ one has $X(\rho) =\lim_{t\rightarrow 0^+} \frac{ \rho (\widetilde\phi_t(p) )- \rho(p)}{t}$, (\ref{2.1}) follows from (\ref{2.7}) and the following lemma. 

\begin{lemma}\label{t3.1a} 
There exist $c_0,\,A_0>0$, depending only on the embedding $s:M\hookrightarrow \mM$, such that for any $x\in s(M)$ and  $0\leq t\leq c_0$, one has
\begin{align}\label{2.5}
 \rho\left(\widetilde\phi_t(x)\right)  \leq A_0t.
\end{align}
\end{lemma}
\begin{proof} Take any  $x\in s(M)$. 
If $t=0$, then (\ref{2.5}) clearly holds. Recall that $\widetilde\phi_t$ maps 
  $\mM_{\pi(p)}$ isometrically to   $\mM_{\phi_t(\pi(p))}$. Thus
one has
\begin{align}\label{2.5a}
\rho\left(\widetilde\phi_t(x)\right)=\rho\left(\widetilde\phi_t^{-1}(s(\phi_t(\pi(x))))\right).
\end{align}
Since $\widetilde\phi_t^{-1}(s(\phi_t(\pi(x))))$ depends smoothly on $t$,  one sees from (\ref{2.5a}) that   (\ref{2.5}) holds at $x\in s(M)$. By the compactness of $s(M)$, it holds for all $x\in s(M)$. 
\end{proof}

To prove (\ref{2.2}), we first observe that by (\ref{1.5x}) one has that for any 
$U \in\Gamma(\mF_2^\perp)$, the following identity holds (cf. (\ref{f7})), 
 \begin{align}\label{2.8}
p_2^\perp \nabla^{T\mM}_UX =0. 
\end{align}

From (\ref{2.8}) and the fact that $[X,\sigma]=[\pi^*Y,\sigma]\in\Gamma(\mF_2^\perp)$ (cf. \cite[Lemma 10.7]{BeGeVe}), one sees that in order to prove (\ref{2.2}), one need only to prove that
 \begin{align}\label{2.9}
|[X,\sigma]|\leq \frac{A_1}{\rho}. 
\end{align}

To prove (\ref{2.9}), recall that   (cf. \cite[Theorem 2.3 of Chapter 6]{CCL})
 \begin{align}\label{2.9a}
[X,\sigma] = \lim_{t\rightarrow 0^+}\frac{\sigma - \left(\widetilde\phi_{t}\right)_*\sigma}{t}. 
\end{align}

Since $\widetilde\phi_t$ maps geodesics in $\mM_{\phi_{-t}(\pi(p))}$ to geodesics in $\mM_{\pi(p)}$, one sees as in \cite[\S 5]{Co86} that at $p\in\mM\setminus s(M)$, $(\widetilde\phi_{t} )_*\sigma$ is the unit vector tangent to the geodesic connecting  $p$ and  $\widetilde\phi_t(s
(\phi_{-t}(\pi(p))))$.

Consider the geodesic triangle in $\mM_{\pi(p)}$ with   vertices $p, \, s(\pi(p))$ and $\widetilde\phi_t(s(\phi_{-t}(\pi(p))))$. Let $\alpha_p$ be the angle at $p$. Then one has
 \begin{align}\label{2.10}
\left| \sigma - \left(\widetilde\phi_{t}\right)_*\sigma\right|^2=2\left(1-\cos \left(\alpha_p\right)\right)  . 
\end{align}

Since $\mM_{\pi(p)}$ is of nonpositive curvature, one has (cf. \cite[Corollary I.13.2]{He}),
 \begin{align}\label{2.11}
\left( \rho\left( \widetilde\phi_t\left(s\left(\phi_{-t}(\pi(p))\right)\right) \right)\right)^2\geq 2\left(1-\cos \left(\alpha_p\right)\right)\rho(p)\,
   d^{\mM_{\pi(p)}}\left( p, \widetilde\phi_t\left(s\left(\phi_{-t}(\pi(p))\right)\right) \right) .
\end{align}
From (\ref{2.10}) and (\ref{2.11}), one gets
 \begin{align}\label{2.12}
\left| \sigma - \left(\widetilde\phi_{t}\right)_*\sigma\right|\leq \frac{
 \rho\left( \widetilde\phi_t\left(s\left(\phi_{-t}(\pi(p))\right)\right) \right) }
{\sqrt{\rho(p)\,
   d^{\mM_{\pi(p)}}\left( p, \widetilde\phi_t\left(s\left(\phi_{-t}(\pi(p))\right)\right) \right)} }.
\end{align}

From (\ref{2.9a}), (\ref{2.12}) and proceed as in Lemma \ref{t3.1a}, one gets (\ref{2.9}).  
\end{proof}

%%%%%%%%%%%%%%%%%%%%%%%%%%%%%%%%%%%%%%%%%%%
\subsection{Sub-Dirac operators and the vanishing on compact subsets}\label{s2.3}
From now on we assume that there is  $\delta>0$  such that $k^F\geq \delta$ over $M$.
We  also assume that $M$ is spin and carries a fixed spin structure, then $\mF\oplus \mF_1^\perp=\pi^*(TM)$ is spin and carries an induced spin structure.   For simplicity, we also assume first that $\mF_2^\perp$ is  oriented and both   $TM$ and $\mF_2^\perp$ are of even rank.

For any $  \beta,\ \varepsilon>0$,  following (\ref{1.8}), let $g_{\beta,\varepsilon}^{T\mM}$ be the deformed  metric of  (\ref{2.1a})  on $\mM$  defined by
the
  orthogonal splitting,
\begin{align}\label{2.20}\begin{split}
       T\mM =   \mF\oplus \mF^\perp_1\oplus \mF^\perp_2,  \ 
\  \  \
g^{T\mM}_{\beta,\varepsilon}= \beta^2   g^{\mF}\oplus\frac{
g^{\mF^\perp_1}}{ \varepsilon^2 }\oplus g^{\mF^\perp_2}.\end{split}
\end{align}

In what follows, we will use the subscripts (or superscripts) $\beta,\ \varepsilon$ to decorate the geometric objects with respect to the deformed metric $g_{\beta,\varepsilon}^{T\mM}$.  It is clear that for any $X\in \mF\oplus\mF_1^\perp$ and $U\in \mF_2^\perp$, $c_{\beta,\varepsilon}(X)$, $c(U)$ and $\widehat c(U)$ act on $S_{\beta,\varepsilon}(\mF\oplus\mF_1^\perp)\widehat\otimes\Lambda^* (\mF_2^\perp )$ and exchange $(S_{\beta,\varepsilon}(\mF\oplus\mF_1^\perp)\widehat\otimes\Lambda^* (\mF_2^\perp ))_\pm$. 

Let $f_1,\,\cdots,\,f_q$ (resp. $h_1,\,\cdots,\,h_{q_1}$; resp. $e_1,\,\cdots,\,e_{q_2}$) be an orthonormal basis of $(\mF,g^{\mF})$ (resp. $(\mF_1^\perp,g^{\mF_1^\perp})$; resp. $ (\mF_2^\perp,g^{\mF_2^\perp})$). 
By proceeding as in \cite[Section 2]{LZ01} and Sections \ref{s1.4}, \ref{s1.5}, we construct  the   sub-Dirac operator (cf. (\ref{1.43}) and (\ref{1.431}), where we take $F$ in (\ref{1.43}) to be $\mF\oplus\mF_1^\perp$, $F^\perp_1$ in (\ref{1.43}) to be zero and $F^\perp_2$ in (\ref{1.43}) to be $\mF_2^\perp$)
\begin{align}\label{2.21}
D_{\mF\oplus\mF_1^\perp,\beta,\varepsilon}:\Gamma\left(S_{\beta,\varepsilon} (\mF\oplus\mF_1^\perp)\widehat\otimes\Lambda^*\left(\mF_2^\perp\right)\right)
\longrightarrow
\Gamma\left(S_{\beta,\varepsilon}(\mF\oplus\mF_1^\perp)\widehat\otimes\Lambda^*\left(\mF_2^\perp\right)\right)
\end{align}
given by
\begin{align}\label{2.36} 
  D_{\mF\oplus\mF_1^\perp,\beta,\varepsilon}=\sum_{i=1}^q \beta^{-1}c_{\beta,\varepsilon}\left(\beta^{-1}f_i\right)\nabla^{\beta,\varepsilon}_{f_i} + \sum_{s=1}^{q_1} \varepsilon\, c_{\beta,\varepsilon}\left(\varepsilon h_s\right)\nabla^{\beta,\varepsilon}_{h_s} 
 +\sum_{j=1}^{q_2}  c \left(e_j\right)\nabla^{\beta,\varepsilon}_{e_j}  ,
\end{align}
where as in (\ref{1.431}),  $\nabla^{\beta,\varepsilon}$ is the canonical connection on $S_{\beta,\varepsilon}(\mF\oplus\mF_1^\perp)\widehat\otimes \Lambda^*(\mF_2^\perp)$ determined by    (\ref{1.42a}) with respect to  $g_{\beta,\varepsilon}^{T\mM}$. In particular, in view of Remark \ref{t1.4}, one has
\begin{align}\label{2.36a} 
  \left[ \nabla^{\beta,\varepsilon} , \widehat c(\sigma)\right]= \widehat c\left(\nabla^{\mF_2^\perp}\sigma\right ).
\end{align}

Let $D_{\mF\oplus\mF_1^\perp,\beta,\varepsilon,\pm}$  be the restrictions of $D_{\mF\oplus\mF_1^\perp,\beta,\varepsilon}$ on  $(S_{\beta,\varepsilon}(\mF\oplus\mF_1^\perp)\widehat\otimes\Lambda^* (\mF_2^\perp ))_\pm$,
then  
\begin{align}\label{2.14}
\left(D_{\mF\oplus\mF_1^\perp,\beta,\varepsilon,+}\right)^*  = D_{\mF\oplus\mF_1^\perp,\beta,\varepsilon,-}. 
\end{align}

 For any $R>0$, denote
\begin{align}\label{2.151}
\mM_{R} =\{p\in \mM:\ \rho(p )\leq
R\}.
\end{align}
Then $\mM_R$ is a smooth manifold with boundary $\partial \mM_R$. 

 Let $f:[0,1]\rightarrow [0,1]$ be a smooth function such that  $f(t)= 0$ for $0\leq t\leq \frac{1}{4}$, while $f(t) =1$ for $   \frac{1}{2}\leq t\leq 1$.  Let $h:[0,1]\rightarrow [0,1]$ be a smooth function such that $h(t)=1$ for $0\leq t\leq \frac{3}{4}$, while $h(t)=0$ for $\frac{7}{8}\leq t\leq 1$. 

Inspired by \cite{BL91} and \cite{Co86}, we make the following deformation of $ D_{\mF\oplus\mF_1^\perp,\beta,\varepsilon}$  on $\mM_R$, which will play a key role in what follows,
\begin{align}\label{2.151x}
 D_{\mF\oplus\mF_1^\perp,\beta,\varepsilon} + \frac{f\left(\frac{\rho}{R}\right)\widehat c(\sigma)}{\beta}.
\end{align}

\begin{rem}\label{t4.0}
The usual deformation from the analytic localization  point of view (such as in \cite{BL91})  deforms $D_{\mF\oplus\mF_1^\perp,\beta,\varepsilon}$ by $ T{\widehat c(\rho\sigma)} $, with $T>0$ being independent of $\beta$ and $\varepsilon$. On the other hand,  $fc(\sigma)$  has  
occured  in \cite{Co86}, where it is viewed as the symbol of a fiberwise Dirac operator. Here we use  ${ f{\widehat c(\sigma)} }/{\beta}$ to deform  $D_{\mF\oplus\mF_1^\perp,\beta,\varepsilon}$, while  Lemma \ref{t2.1} allows us to get the needed estimates  given in the following  Lemma.
\end{rem}

\begin{lemma}\label{t4.1}
There exists  $R_0>0$ such that for any (fixed) $R\geq R_0$, when $\beta,\ \varepsilon>0$  (which may depend on $R$) are  small enough,

\noindent  (i) for any   $s\in\Gamma(S_{\beta,\varepsilon} (\mF\oplus\mF_1^\perp)\widehat\otimes\Lambda^* (\mF_2^\perp ))$ supported in $\mM_{{R}{}}$, one has\footnote{The norms below  denpend on $\beta$ and $\varepsilon$. In case of no confusion, we omit the subscripts for simplicity.}
\begin{align}\label{4.1}
\left\|\left(D_{\mF\oplus\mF_1^\perp,\beta,\varepsilon} + \frac{f\left(\frac{\rho}{R}\right)\widehat c(\sigma)}{\beta}\right)s\right\|\geq  \frac{\sqrt{ {\delta} }}{4\beta}\,\|s\|;
\end{align}
(ii) for any  $s\in\Gamma(S_{\beta,\varepsilon} (\mF\oplus\mF_1^\perp)\widehat\otimes\Lambda^* (\mF_2^\perp ))$ supported in $\mM_{{R}{}}\setminus \mM_{\frac{R}{2}}$, one has 
\begin{align}\label{4.2}
\left\| \left(  h\left(\frac{\rho}{R}\right)  D_{\mF\oplus\mF_1^\perp,\beta,\varepsilon}  h\left(\frac{\rho}{R}\right)  +
 \frac{  \widehat c(\sigma)}{\beta}\right)   s \right\|\geq  \frac{ { {1} }}{2\beta}\,\|s\|.
\end{align}
\end{lemma}
\begin{proof}

In view of Remark \ref{t1.4} and (\ref{2.36}), one has
\begin{multline}\label{2.37} 
 \left(   D_{\mF\oplus\mF_1^\perp,\beta,\varepsilon} 
  +
 \frac{ f\left(\frac{\rho}{R}\right) \widehat c(\sigma)}{\beta}\right)    ^2
=D_{\mF\oplus\mF_1^\perp,\beta,\varepsilon} ^2+
\frac{f'\left(\frac{\rho}{R}\right)}{\beta R} c_{\beta,\varepsilon}(d\rho) \widehat c(\sigma)
\\
+ \frac{  f\left(\frac{\rho}{R} \right)  }{ \beta } \left[D_{\mF\oplus\mF_1^\perp,\beta,\varepsilon} , \widehat c(\sigma) \right]
+\frac{  f\left(\frac{\rho}{R}\right) ^2}{\beta^2}
 ,
\end{multline}
where we identify $d\rho$ with the gradient of $\rho$. 

By definition, one has on $\mM\setminus s(M)$ that
\begin{align}\label{2.38a} 
 c_{\beta,\varepsilon }(d\rho)= \sum_{i=1}^q \beta^{-1}c_{\beta,\varepsilon}\left(\beta^{-1}f_i\right) {f_i}(\rho) + \sum_{s=1}^{q_1} \varepsilon\, c_{\beta,\varepsilon}\left(\varepsilon h_s\right) {h_s} (\rho)
 +\sum_{j=1}^{q_2}  c \left(e_j\right) {e_j} (\rho) .
\end{align}

By  (\ref{2.36})  and (\ref{2.36a}),  one has on $\mM\setminus s(M)$ that
\begin{multline}\label{2.38b} 
 \left[D_{\mF\oplus\mF_1^\perp,\beta,\varepsilon} , \widehat c(\sigma) \right] 
=\sum_{i=1}^q \beta^{-1}c_{\beta,\varepsilon}\left(\beta^{-1}f_i\right)\widehat c\left(\nabla^{\mF_2^\perp}_{f_i}\sigma\right) + \sum_{s=1}^{q_1} \varepsilon\, c_{\beta,\varepsilon}\left(\varepsilon h_s\right) \widehat c\left(\nabla^{\mF_2^\perp}_{h_s}\sigma\right) 
\\
 +\sum_{j=1}^{q_2}  c \left(e_j\right) \widehat c\left(\nabla^{\mF_2^\perp}_{e_j}\sigma\right) .
\end{multline}

By Lemma \ref{t2.1}, (\ref{2.38a})  and (\ref{2.38b}), we find that there exists a constant $C>0$, not depending   on 
$R,\, \beta,\, \varepsilon>0$, such that the following inequality holds on $\mM_R\setminus s(M)$,
\begin{align}\label{2.38} 
\frac{\left|c_{\beta,\varepsilon }(d\rho)\right|}{R}   + f\left(\frac{\rho}{R} \right)\left|
 \left[D_{\mF\oplus\mF_1^\perp,\beta,\varepsilon} , \widehat c(\sigma) \right] \right|\leq \frac{C}{\beta R} +O_R\left(1\right),
\end{align}
where by $O_R(\cdot)$ we mean that the estimating constant might depend on $R>0$.

On the other hand,  by (\ref{1.50}), the following formula holds on $\mM_R$, 
\begin{align}\label{2.39} 
 D_{\mF\oplus\mF_1^\perp,\beta,\varepsilon}^2=-\Delta^{\beta,\varepsilon}+\frac{k^\mF}{4\beta^2}+O_R\left(\frac{1}{\beta}+\frac{\varepsilon^2}{\beta^2}\right)
,
\end{align}
where $-\Delta^{\beta,\varepsilon}\geq 0$ is the corresponding Bochner Laplacian, and $k^\mF=\pi^*k^F\geq \delta$.

From   (\ref{2.37}), (\ref{2.38}) and (\ref{2.39}), one sees that if one first fixes a sufficiently large $R>0$ and then makes  $\beta>0$, $\varepsilon>0$ sufficiently small,  one deduces (\ref{4.1}) easily.

Now  by  (\ref{2.36}) one has  on $\mM_R\setminus s(M)$ that
 \begin{multline}\label{2.44} 
 \left(  h\left(\frac{\rho}{R}\right)  D_{\mF\oplus\mF_1^\perp,\beta,\varepsilon}  h\left(\frac{\rho}{R}\right) 
  +
 \frac{ \widehat c(\sigma)}{\beta}\right)    ^2
=\left( h\left(\frac{\rho}{R}\right)  D_{\mF\oplus\mF_1^\perp,\beta,\varepsilon}  h\left(\frac{\rho}{R}\right) \right)^2
\\
+ \frac{  h\left(\frac{\rho}{R}\right) ^2}{\beta} \left[ D_{\mF\oplus\mF_1^\perp,\beta,\varepsilon}, \widehat c(\sigma)\right]
+\frac{ 1}{\beta^2}
 . 
\end{multline}

From    (\ref{2.38}) and (\ref{2.44}),   one gets (\ref{4.2}), where ${\rm Supp}(s)\subseteq \mM_R\setminus\mM_{\frac{R}{2}}$, similarly.  
\end{proof}

Lemma \ref{t4.1}  motivates the proof of Theorem \ref{t0.2} (for the case of $\dim M=4k$) given in  Section \ref{s2.4}, where we  make use of a trick of Braverman \cite[\S 14]{Brav02} (See also \cite[\S 3]{MZ12}).  This approach reflects the topological nature of the $\widehat A$-genus and the involved indices.

%%%%%%%%%%%%%%%%%%%%%%%%%%%%%%%%%%%%%%%%%%%
\subsection{Proof of Theorem \ref{t0.2} for the case of $\dim M=4k$}\label{s2.4}

Let $\partial \mM_R$ bound another oriented  manifold $\mN _R$ so that $\widetilde \mN_R=\mM_R\cup\mN_R$ is a closed manifold (for example, one can take the double of $\mM_R$). 
Let $E$ be a Hermitian vector bundle over $\mM_R$ such that $(S_{\beta,\varepsilon}(\mF\oplus\mF_1^\perp)\widehat\otimes\Lambda^*(\mF_2^\perp))_-\oplus E$
is a trivial vector bundle over $\mM_R$. Then $(S_{\beta,\varepsilon}(\mF\oplus\mF_1^\perp)\widehat\otimes\Lambda^*(\mF_2^\perp))_+\oplus E$ is a trivial vector bundle near $\partial\mM_R$, under the identification $\widehat c(\sigma)+{\rm Id}_E$. 

By extending obviously the above trivial vector bundles to $\mN_R$, we get a ${\bf Z}_2$-graded Hermitian vector bundle $\xi=\xi_+\oplus\xi_-$ over $\widetilde\mN_R$ and an odd self-adjoint endomorphism $V=v+v^*\in\Gamma({\rm End}(\xi))$ (with $v:\Gamma(\xi_+)\rightarrow\Gamma(\xi_-)$, $v^*$ being the adjoint of $v$) such that 
\begin{align}\label{2.19a}
\xi_\pm  = \left (S_{\beta,\varepsilon}\left(\mF\oplus\mF_1^\perp\right)\widehat\otimes\Lambda^*\left(\mF_2^\perp\right)\right)_\pm\oplus E 
\end{align}
over $\mM_R$, $V$ is invertible on $\mN_R$ and 
\begin{align}\label{2.19b}
V=f\left(\frac{\rho}{R}\right)\widehat c(\sigma)+{\rm Id}_E 
\end{align}
on $\mM_R$, which is invertible on $\mM_R\setminus \mM_{\frac{R}{2}}$. 

 Recall  that  $h(\frac{\rho}{R})$   vanishes near $\partial \mM_R$. We extend it to a function on $\widetilde\mN_R$ which equals to zero on $\mN_R$, and denote the resulting function on $\widetilde\mN_R$ by $\widetilde h_R$.  
Let $\pi_{\widetilde\mN_R}:T\widetilde\mN_R\rightarrow \widetilde\mN_R$ be the projection of the tangent bundle of $\widetilde\mN_R$. 
Let $\gamma  ^{\widetilde\mN_R}\in{\rm Hom}(\pi_{\widetilde\mN_R}^*\xi_+,\pi_{\widetilde\mN_R}^*\xi_-)$ be the symbol defined by
\begin{align}\label{2.19c}
\gamma^{\widetilde\mN_R}(p,w)= \pi_{\widetilde\mN_R}^*\left(\sqrt{-1}\,\widetilde h^2_R\,c _{\beta,\varepsilon} (w)+v(p)\right),\ \ {\rm for}
\ \ p\in\widetilde\mN_R,\ \ w\in T_p\widetilde\mN_R. 
\end{align}
By (\ref{2.19b}) and (\ref{2.19c}), $\gamma  ^{\widetilde\mN_R}$ is singular only if $w=0$ and $p\in\mM_{\frac{R}{2}}$. Thus $\gamma  ^{\widetilde\mN_R}$ is an  elliptic symbol.

On the other hand,  it is clear that $\widetilde h_R D_{\mF\oplus\mF_1^\perp,\beta,\varepsilon} \widetilde h_R$ is well-defined on $\widetilde \mN_R$ if we define it to equal to zero on $\widetilde \mN_R\setminus\mM_R$.  

Let $A:L^2(\xi)\rightarrow L^2(\xi)$ be a second order positive elliptic differential operator on $\widetilde\mN_R$  preserving the ${\bf Z}_2$-grading of $\xi=\xi_+\oplus\xi_-$, such that its symbol equals to $|\eta|^2$ at $\eta\in T\widetilde\mN_R$ (to be more precise, here   $A$ also depends on the defining metric. We omit the corresponding subscript/superscript only for convenience). 
Let $P_{R,\beta,\varepsilon}:L^2(\xi)\rightarrow L^2(\xi)$ be the zeroth order   pseudodifferential operator on $\widetilde \mN_R$ defined by
\begin{align}\label{2.22}
P_{R,\beta,\varepsilon}=A^{-\frac{1}{4}} \widetilde h_R D_{\mF\oplus\mF_1^\perp,\beta,\varepsilon} \widetilde h_RA^{-\frac{1}{4} }
+\frac{V}{\beta}. 
\end{align}
Let $P_{R,\beta,\varepsilon,+}:L^2(\xi_+)\rightarrow L^2(\xi_-)$ be the obvious restriction. Then   the principal symbol of $P_{R,\beta,\varepsilon,+}$, which we denote by $\gamma(P_{R,\beta,\varepsilon,+})$, is  homotopic through elliptic symbols to 
$\gamma ^{\widetilde\mN_R}$. 
Thus $P_{R,\beta,\varepsilon,+}$ is a Fredholm operator. 
Moreover, by the Atiyah-Singer index theorem \cite{ASI} (cf. \cite[Theorem 13.8 of Chap. III]{LaMi89}),  one finds
\begin{align}\label{2.23}
{\rm ind}\left(P_{R,\beta,\varepsilon,+}\right) =  
\widehat A(M).
\end{align}

Inspired by \cite[\S 14]{Brav02} (See also \cite[\S 3]{MZ12}), for any $0\leq t\leq 1$, set 
\begin{align}\label{2.24}
P_{R,\beta,\varepsilon,+}(t) =A^{-\frac{1}{4}} \widetilde h_R D_{\mF\oplus\mF_1^\perp,\beta,\varepsilon} \widetilde h_RA^{-\frac{1}{4} }
+\frac{tv}{\beta} +A^{-\frac{1}{4}} \frac{(1-t)v}{\beta} A^{-\frac{1}{4}}  . 
\end{align}
Then   $ P_{R,\beta,\varepsilon,+}(t) $ is a smooth family of zeroth order  pseudodifferential operators such that the corresponding symbol $\gamma( P_{R,\beta,\varepsilon,+}(t) )$ is elliptic 
 for $0<t\leq 1$. 
Thus $P_{R,\beta,\varepsilon,+}(t) $ is a continuous family of Fredholm operators for $0<t\leq 1$ with $P_{R,\beta,\varepsilon,+}(1)=P_{R,\beta,\varepsilon,+} $. 

Now since $P_{R,\beta,\varepsilon,+}(t) $ is continuous on the whole $[0,1]$, in view of  (\ref{2.23}), if $P_{R,\beta,\varepsilon,+}(0) $ is Fredholm and has vanishing index, then  Theorem \ref{t0.2} follows from (\ref{2.23}). 

Thus we need only to prove the following result. 

\begin{prop}\label{t2.3}  
There exist $R,\ \beta,\ \varepsilon>0$ such that the following identity holds,
\begin{align}\label{2.25}
\dim\left(\ker\left(P_{R,\beta,\varepsilon,+}(0)  \right)\right) =\dim\left(\ker\left(P_{R,\beta,\varepsilon,+}(0) ^* \right)\right) =0. 
\end{align}
\end{prop}

\begin{proof}
By definition, $P_{R,\beta,\varepsilon}(0):L^2(\xi)\rightarrow L^2(\xi)$ is given by 
\begin{align}\label{2.26}
P_{R,\beta,\varepsilon}(0) =A^{-\frac{1}{4}} \widetilde h_R D_{\mF\oplus\mF_1^\perp,\beta,\varepsilon} \widetilde h_RA^{-\frac{1}{4} }+
 A^{-\frac{1}{4}} \frac{V}{\beta} A^{-\frac{1}{4}}  . 
\end{align}

By (\ref{2.14}),  $P_{R,\beta,\varepsilon}(0)$ is  formally  self-adjoint. Thus we need to show that 
\begin{align}\label{2.27}
\dim\left(\ker\left(P_{R,\beta,\varepsilon}(0)  \right)\right)  =0
\end{align}
for certain $R,\ \beta,\ \varepsilon>0$. 
Let $s\in \ker(P_{R,\beta,\varepsilon}(0) )$. By (\ref{2.26}) one has 
\begin{align}\label{2.28}
\left( \widetilde h_R D_{\mF\oplus\mF_1^\perp,\beta,\varepsilon} \widetilde h_R +
   \frac{V}{\beta}\right) A^{-\frac{1}{4}}  s=0. 
\end{align}

Since $\widetilde h_R=0$  on $\widetilde \mN_R\setminus \mM_R$, while $V$ is invertible on $\widetilde \mN_R\setminus \mM_R$, one has by (\ref{2.28}) 
\begin{align}\label{2.29}
A^{-\frac{1}{4}}  s=0\ \ \ \ \ \ \ {\rm on}\ \ \ \ \ \widetilde \mN_R\setminus \mM_R. 
\end{align}

Write on $\mM_R$ that 
\begin{align}\label{2.30}
A^{-\frac{1}{4}}  s= s_1+s_2, 
\end{align}
with $s_1\in L^2 (S_{\beta,\varepsilon}(\mF\oplus\mF_1^\perp)\widehat\otimes\Lambda^* (\mF_2^\perp))$ and $s_2\in L^2(E)$. 

By (\ref{2.19b}), (\ref{2.28}) and (\ref{2.30}),  one has
\begin{align}\label{2.31}
s_2=0 ,
\end{align}
while
\begin{align}\label{2.32} 
\left( \widetilde h_R D_{\mF\oplus\mF_1^\perp,\beta,\varepsilon} \widetilde h_R +
 \frac{ f\left(\frac{\rho}{R}\right) \widehat c(\sigma)}{\beta}\right)   s_1=0. 
\end{align}

We need to show that (\ref{2.32}) implies $s_1=0$. 

Let $\alpha:[0,1]\rightarrow [0,1]$ be a smooth function such that $\alpha(t)=0$ for $0\leq t\leq \frac{1}{2}$, while $\alpha(t)=1$ for $\frac{2}{3}\leq t\leq 1$. 

Following \cite[pp. 115]{BL91}, 
let $\alpha_1$, $\alpha_2$ be the smooth functions on $\mM_R$ defined by
\begin{align}\label{2.33} 
\alpha_1=\frac{1- \alpha\left(\frac{\rho}{R}\right)}{\left(  \alpha\left(\frac{\rho}{R}\right)^2+\left(1- \alpha\left(\frac{\rho}{R}\right)\right)^2\right)^{\frac{1}{2}} } ,\ \ \ \alpha_2=\frac{ \alpha\left(\frac{\rho}{R}\right)}{\left(  \alpha\left(\frac{\rho}{R}\right)^2+\left(1- \alpha\left(\frac{\rho}{R}\right)\right)^2\right)^{\frac{1}{2}} } .
\end{align}
Then $\alpha_1^2+\alpha_2^2=1$   on $\mM_R$. Clearly,
  $\alpha_1\widetilde{h}_R=\alpha_1$, $\alpha_2f(\frac{\rho}{R})=\alpha_2$.  Thus, one has
\begin{multline}\label{2.34} 
\left\|\left( \widetilde h_R D_{\mF\oplus\mF_1^\perp,\beta,\varepsilon} \widetilde h_R +
 \frac{ f\left(\frac{\rho}{R}\right) \widehat c(\sigma)}{\beta}\right)   s_1\right\|^2 
= \left\|\alpha_1\left(   D_{\mF\oplus\mF_1^\perp,\beta,\varepsilon} 
  +
 \frac{ f\left(\frac{\rho}{R}\right) \widehat c(\sigma)}{\beta}\right)   s_1\right\|^2 
\\
+\left\|\alpha_2\left( \widetilde h_R D_{\mF\oplus\mF_1^\perp,\beta,\varepsilon} \widetilde h_R +
 \frac{  \widehat c(\sigma)}{\beta}\right)   s_1\right\|^2  ,
\end{multline}
from which  one gets
\begin{multline}\label{2.35} 
\sqrt{2}\, \left\|\left( \widetilde h_R D_{\mF\oplus\mF_1^\perp,\beta,\varepsilon} \widetilde h_R +
 \frac{ f\left(\frac{\rho}{R}\right) \widehat c(\sigma)}{\beta}\right)   s_1\right\|
\geq  \left\|\alpha_1\left(   D_{\mF\oplus\mF_1^\perp,\beta,\varepsilon} 
  +
 \frac{ f\left(\frac{\rho}{R}\right) \widehat c(\sigma)}{\beta}\right)   s_1\right\| 
\\
+\left\|\alpha_2\left( \widetilde h_R D_{\mF\oplus\mF_1^\perp,\beta,\varepsilon} \widetilde h_R +
 \frac{  \widehat c(\sigma)}{\beta}\right)   s_1\right\| 
\geq
 \left\|\left(   D_{\mF\oplus\mF_1^\perp,\beta,\varepsilon} 
  +
 \frac{ f\left(\frac{\rho}{R}\right) \widehat c(\sigma)}{\beta}\right)  \left(\alpha_1 s_1\right)\right\| 
\\
+\left\|\left( \widetilde h_R D_{\mF\oplus\mF_1^\perp,\beta,\varepsilon} \widetilde h_R +
 \frac{  \widehat c(\sigma)}{\beta}\right)  \left(\alpha_2 s_1\right)\right\| 
-\left\|  c_{\beta,\varepsilon}\left( d\alpha_1\right)s_1 \right\| 
-\left\|  c_{\beta,\varepsilon}\left( d\alpha_2\right)s_1 \right\| 
 ,
\end{multline}
where for each $i\in\{1,\,2\}$, we identify $d\alpha_i$   with the gradient of $\alpha_i$.

By Lemma \ref{t2.1},   (\ref{2.38a})  and  (\ref{2.33}),  there is   $C_1>0$,   not depending    on 
$R,\, \beta,\, \varepsilon>0$, such that
\begin{align}\label{2.41} 
\left|  c_{\beta,\varepsilon}\left( d\alpha_1\right) \right| +\left|  c_{\beta,\varepsilon}\left( d\alpha_2\right) \right| \leq \frac{C_1}{\beta R}+O_R(1) .
\end{align}

From  Lemma \ref{t4.1},    (\ref{2.35}) and (\ref{2.41}), 
one finds that there exist $   R,\ \beta,\ \varepsilon>0$   
such that
\begin{align}\label{2.47} 
\left\|\left( \widetilde h_R D_{\mF\oplus\mF_1^\perp,\beta,\varepsilon} \widetilde h_R +
 \frac{ f\left(\frac{\rho}{R}\right) \widehat c(\sigma)}{\beta}\right)   s_1\right\|
\geq  \frac{ \left\|s_1\right\|}{ \sqrt{ \beta} }  . 
\end{align}

From (\ref{2.28})-(\ref{2.32}),  (\ref{2.47}) and the invertibility of $A^{-\frac{1}{4}}$, one sees that for suitable $R,\ \beta,\ \varepsilon>0$, (\ref{2.27}) holds. This completes the proof of Proposition \ref{t2.3}, 
which implies   Theorem \ref{t0.2} for the case of $\dim M=4k$, when $\mF_2^\perp$ is orientable and of even rank.  
\end{proof}

If ${\rm rk}(\mF^\perp_2)$ is not even, we can consider $M\times M\times M\times M$  to make it even. If $ \mF_2^\perp$ is not orientable, then we can consider the double covering of $M$ with respect to $w_1(\mF_2^\perp)$, the first Stiefel-Whitney class of $\mF_2^\perp$, and consider the pull-back of $\mF_2^\perp$ on the double covering.
The proof of Theorem \ref{t0.2} for the case of $\dim M=4k$ is thus completed.

\begin{rem}\label{t44}
One may also  use $\frac{\rho}{R}$ instead of $f(\frac{\rho}{R} )$    in the above proof. 
\end{rem}

%%%%%%%%%%%%%%%%%%%%%%%%%%%%%%%%%%%%%%%%%%%
\subsection{The case of the mod $2$ index}\label{s2.5}

In this subsection, we consider the cases of $\dim M=8k+i$, $i=1,\,2$.  Here we deal with the case of $\dim M=8k+1$, where one considers {real}   operators as in \cite{ASV},  in detail.   
By multiplying $M$ by a Bott manifold of dimension $8$, which is a compact spin manifold $B^8$ such that $\widehat A(B^8)=1$,  we may well assume that $q_1>1$.  Then    
$\partial\mM_R$  
 is connected.

Let $f_1,\,\cdots,\,f_{q+q_1}$   be an oriented orthonormal basis of $(\mF\oplus \mF_1^\perp,{\beta^2g^\mF\oplus\frac{g^{\mF_1^\perp}}{\varepsilon^2}})$. 
Set
\begin{align}\label{2.47a}
\tau _{\beta,\varepsilon}= c_{\beta,\varepsilon}\left(f_1\right)\cdots c_{\beta,\varepsilon}\left(f_{q+q_1}\right).
\end{align}
Let $\widehat\tau$ be the ${\bf Z}_2$-grading operator for $
\Lambda^*(\mF_2^\perp)= \Lambda^{\rm even}(\mF_2^\perp)\oplus 
 \Lambda^{ {\rm odd}}(\mF_2^\perp)$.

  Inspired by  \cite[\S 3]{ASV} and \cite[(3.1)]{BZ93} (compare with \cite{Z93} which deals with  the case of $\dim M=8k+2$), we modify the sub-Dirac operator in  (\ref{2.21})   
by
\begin{align}\label{2.48}
 \widehat \tau\,\tau_{\beta,\varepsilon}  D_{\mF\oplus\mF_1^\perp,\beta,\varepsilon}
:\Gamma\left(S_{\beta,\varepsilon}(\mF\oplus\mF_1^\perp) \otimes\Lambda^*\left(\mF_2^\perp\right)\right)
\longrightarrow
\Gamma\left(S_{\beta,\varepsilon}(\mF\oplus\mF_1^\perp) \otimes\Lambda^*\left(\mF_2^\perp\right)\right),
\end{align}
which is formally skew-adjoint (here by dimension reason there is no ${\bf Z}_2$-grading of the real spinor bundle $S_{\beta,\varepsilon}(\mF\oplus\mF_1^\perp)$). 
We also   modify $V=v+v^*$ in (\ref{2.19b}) by 
\begin{align}\label{2.49}
\widehat V=\widehat v-\widehat v^*
\end{align}
such that one has, on $\mM_R$,  the following formula for $\widehat v$ acting between real vector bundles,
\begin{multline}\label{2.491}
\widehat  v=f\left(\frac{\rho}{R}\right) \widehat\tau\widehat c(\sigma)+{\rm Id}_E
: \Gamma\left(S_{\beta,\varepsilon}(\mF\oplus\mF_1^\perp) \otimes\Lambda^{\rm even}\left(\mF_2^\perp\right)\oplus E\right)
\\
\longrightarrow 
 \Gamma\left(S_{\beta,\varepsilon}(\mF\oplus\mF_1^\perp) \otimes\Lambda^{\rm odd}\left(\mF_2^\perp\right)\oplus E\right)
.
\end{multline}

We then modify the operator $P_{R,\beta,\varepsilon}$ in (\ref{2.22}) by 
\begin{align}\label{2.50}
\widehat P_{R,\beta,\varepsilon}= A^{-\frac{1}{4}} \widetilde h_R \tau_{\beta,\varepsilon}  \widehat \tau\, D_{\mF\oplus\mF_1^\perp,\beta,\varepsilon}
 \widetilde h_RA^{-\frac{1}{4} }
+\frac{\widehat V}{\beta},
\end{align}
which is clearly formally skew-adjoint. 
By direct computation, one has
\begin{align}\label{2.51}
\left(  \widehat\tau\widehat c(\sigma) \right)^* =\widehat c(\sigma) \widehat \tau  =- \widehat\tau\widehat c(\sigma) 
\end{align}
and that for any $X\in T\mM$, 
\begin{align}\label{2.52}
 \widehat \tau\,  \tau c(X)\widehat\tau\widehat c(\sigma) + \widehat\tau\widehat c(\sigma) \widehat\tau\,\tau c(X)=\tau c(X) \widehat c(\sigma)-   \widehat c(\sigma) \tau  c(X)=0.
\end{align}

From (\ref{2.50})-(\ref{2.52}), one sees that $(\widehat P_{R,\beta,\varepsilon})^2$ has an elliptic symbol. Thus $\widehat P_{R,\beta,\varepsilon}$ is a zeroth order real skew-adjoint  elliptic pseudodifferential operator, and  thus admits a mod 2 index in the sense of   \cite{ASV}. Moreover,   by the mod $2$ index theorem in \cite{ASV} (cf. \cite{LaMi89}), one has
\begin{align}\label{2.53a}
\alpha(M)= \dim\left(\ker\left(\widehat P_{R,\beta,\varepsilon}\right)\right)\  {\rm mod} \ 2.
\end{align}

Now by proceeding as in Section \ref{s2.4}, one sees that there are $R,\,\beta,\,\varepsilon>0$ such that
\begin{align}\label{2.53}
\dim\left(\ker\left(\widehat P_{R,\beta,\varepsilon}\right)\right)= 0.
\end{align}

From (\ref{2.53a}) and  (\ref{2.53}), one gets $\alpha (M)=0$.

%%%%%%%%%%%%%%%%%%%%%%%%%%%%%%%%%%%%%%%%%%%
\subsection{Proof of  the Connes vanishing theorem  and more}\label{s2.6}

Without loss of generality, we may and we will assume that all $\mF=\pi^*F$, $\mF^\perp_1$ and $\mF_2^\perp$ are oriented and of even rank. The main concern here is that we only assume $F$ is spin, not $TM$. Thus, here $\mF=\pi^*F$ is spin and carries a fixed spin structure. 

Instead of the sub-Dirac operator considered in (\ref{2.21}), 
we now consider the sub-Dirac operator   constructed as in (\ref{1.43}),
\begin{multline}\label{2.54}
D^{\mF,\phi(\mF_1^\perp)}_{\beta,\varepsilon}
:\Gamma\left(S(\mF)\widehat \otimes\Lambda^*\left(\mF_1^\perp\right) \widehat\otimes\Lambda^*
\left(\mF_2^\perp\right)\otimes\phi\left(\mF_1^\perp\right)\right)
\\
\longrightarrow 
\Gamma\left(S(\mF)\widehat \otimes\Lambda^*\left(\mF_1^\perp\right)\widehat\otimes\Lambda^*
\left(\mF_2^\perp\right)\otimes\phi\left(\mF_1^\perp\right)\right).
\end{multline}

Now we can proceed as in Sections \ref{s2.3} and \ref{s2.4}, by replacing the sub-Dirac operator in (\ref{2.21}) by $D^{\mF,\phi(\mF_1^\perp)}_{\beta,\varepsilon}$ above.

In particular, by the Atiyah-Singer index theorem \cite{ASI}, the right hand side of the formula corresponding to  (\ref{2.23}) is now
\begin{align}\label{2.55z}
 2^{\frac{q_1}{2}}
\left\langle\widehat A(F)\widehat L(TM/F){\rm ch}(\phi(TM/F)),[M]\right\rangle.
\end{align}

In summary,   if  $k^F$ is positive over $M$, then we get
\begin{align}\label{2.55}
\left\langle\widehat A(F)\widehat L(TM/F){\rm ch}(\phi(TM/F)),[M]\right\rangle=0.
\end{align}

Now as any rational Pontrjagin class of $TM/F$ can be expressed as a rational linear combination of classes of form $\widehat L(TM/F){\rm ch}(\phi(TM/F))$, one gets from (\ref{2.55}) that for any Pontrjagin class $p(TM/F)$ of $TM/F$, one has 
\begin{align}\label{2.55a}
  \left\langle\widehat A(F) p(TM/F) ,[M]\right\rangle=0,
\end{align}
which has   been proved in \cite[Corollary 8.3]{Co86}. 
In particular, one has
\begin{align}\label{2.56}
\widehat A(M)=  \left\langle\widehat A(TM)  ,[M]\right\rangle=\left\langle\widehat A(F)\widehat A(TM/F) ,[M]\right\rangle=0,
\end{align}
which completes
the proof of Theorem \ref{t0.1}.

\begin{rem} \label{t3.1} 
If one modifies the sub-Dirac operator in (\ref{2.21}) by twisting an integral power of $ \mF_1^\perp$, then one sees that (\ref{2.55a}) also holds under the condition of Theorem \ref{t0.2}. This generalizes \cite[Theorem 3.1]{LZ01}. 
\end{rem}

 By further modifying the sub-Dirac operators involved above, one   gets the following generalization of  Theorems \ref{t0.2} and \ref{t0.1} (compare with \cite[Theorem 3.2]{LZ01}).

\begin{thm} \label{t3.2} 
Under the assumptions of either Theorem \ref{t0.2} or \ref{t0.1}, if $TM/F$ is also oriented, then for any Pontrjagin class $p(TM/F)$ of $TM/F$, one has for any integer $k\geq 0$ that
\begin{align}\label{2.57}
  \left\langle\widehat A(F)p(TM/F) e(TM/F)^k ,[M]\right\rangle= 0.
\end{align}
In particular, 
\begin{align}\label{2.58}
  \left\langle\widehat A(F)e(TM/F)  ,[M]\right\rangle= 0.
\end{align}
\end{thm}

%%%%%%%%%%%%%%%%%%%%

Under the assumption of Theorem \ref{t3.2}, if one assumes that $\dim M=6$ and ${\rm rk}(F)=4$, then by (\ref{2.57}) one gets
 \begin{align}\label{2.59}
\left\langle  e(TM/F)^3,[M]\right\rangle= 0.
\end{align}

From (\ref{2.59}), one obtains the following   partial complement to a classical result of Bott \cite[Corollary 1.7]{Bo70} which states that there is no smooth codimension two foliation on the complex projective space ${\bf C}P^{2n+1}$ with $n\geq 2$. 

\begin{cor}\label{t3.4} There is no smooth codimension two foliation of positive leafwise scalar curvature on ${\bf C}P^3$. 
\end{cor}

%%%%%%%%%%%%%%%%%%%%%%%%%%%

$\ $

\noindent{\bf Acknowledgements.} The author  is indebted  to Kefeng
Liu for sharing his ideas in the joint work \cite{LZ01} and for
many related discussions.  The author is also grateful to Huitao
Feng, Xiaonan Ma and Yong Wang  for  many helpful suggestions. Last but not least, the author thanks the referees for critical reading and helpful comments and suggestions. This work was
partially supported by MOEC and NNSFC.

\def\cprime{$'$} \def\cprime{$'$}
\providecommand{\bysame}{\leavevmode\hbox to3em{\hrulefill}\thinspace}
\providecommand{\MR}{\relax\ifhmode\unskip\space\fi MR }
 
\providecommand{\MRhref}[2]{%
  \href{http://www.ams.org/mathscinet-getitem?mr=#1}{#2}
}
\providecommand{\href}[2]{#2}

\end{document}